\documentclass{compositio}

\pdfoutput=1

\usepackage{amsmath}
\usepackage{comment}

\usepackage{amsfonts}
\usepackage{amssymb}
\usepackage{amsthm}
\usepackage{bbm}
\usepackage{color}
\usepackage{epsfig}
\usepackage{mathrsfs}
\usepackage{amscd}
\usepackage[all,color]{xy}
\usepackage{rotating}
\usepackage{lscape}
\usepackage{amsbsy}
\usepackage{verbatim}
\usepackage{moreverb}
\usepackage{enumerate}
\usepackage{stmaryrd}
\usepackage{mathdots}

\usepackage[pagebackref]{hyperref}
\usepackage[nameinlink,capitalise,noabbrev]{cleveref}

\hypersetup{%
  bookmarksnumbered=true,%
  bookmarks=true,%
  colorlinks=true,%
  linkcolor=blue,%
  citecolor=blue,%
  filecolor=blue,%
  menucolor=blue,%
  pagecolor=blue,%
  urlcolor=blue,%
  pdfnewwindow=true,%
  pdfstartview=FitBH}

\Crefname{figure}{Figure}{Figures}

\newtheorem{theorem}{Theorem}[section]
\newtheorem{prop}[theorem]{Proposition}
\newtheorem{lemma}[theorem]{Lemma}
\newtheorem{cor}[theorem]{Corollary}

\theoremstyle{definition}
\newtheorem{definition}[theorem]{Definition}

\newtheorem{exercise*}[exercise]{$\star$ Exercise}

\theoremstyle{remark}
\newtheorem{remark}[theorem]{Remark}

\newtheorem{example}[theorem]{Example}
\newtheorem{question}[theorem]{Question}

\DeclareMathOperator*{\colim}{colim}
\DeclareMathOperator{\cone}{cone}

\DeclareMathOperator*{\hocolim}{hocolim}

\DeclareMathOperator{\res}{Res}
\DeclareMathOperator{\Spec}{Spec}
\DeclareMathOperator{\tel}{tel}
\DeclareMathOperator{\Ind}{Ind}
\DeclareMathOperator{\Sub}{Sub}

\newcommand{\C}{\mathbb{C}}
\newcommand{\R}{\mathbb{R}}
\newcommand{\Z}{\mathbb{Z}}
\newcommand{\Q}{\mathbb{Q}}
\newcommand{\T}{\mathbb{T}}
\newcommand{\N}{\mathbb{N}}

\newcommand{\xra}{\xrightarrow}

\def\cC{\mathcal C}\def\cD{\mathcal D}

\def\cS{\mathcal S}
\def\cX{\mathcal X}

\def\cC{\mathcal C}\def\cD{\mathcal D}

\def\cS{\mathcal S}
\def\cX{\mathcal X}

\newcommand{\type}[1]{\mathrm{type}_{#1}}
\newcommand{\subjconj}[1]{\le^G\!#1}
\newcommand{\tti}[1]{\mathrm{thick}^{\otimes}({#1})}

\DeclareMathOperator{\Alg}{Alg}
\DeclareMathOperator{\supp}{supp}
\DeclareMathOperator{\Mod}{Mod}
\DeclareMathOperator{\Loc}{Loc}
\DeclareMathOperator{\Spc}{Spc}

\newcommand{\SH}{\mathcal{SH}}
\newcommand{\NN}{\overline{\N}}
\newcommand{\F}{\mathbb{F}}

\begin{document}

\title{On the Balmer spectrum for compact Lie groups}

\author{Tobias Barthel}
\email{tbarthel@math.ku.dk}
\address{Department of Mathematical Sciences, University of Copenhagen, Universitetsparken 5, 2100 K{\o}benhavn {\O}, Denmark}

\author{J.P.C.~Greenlees}
\email{john.greenlees@warwick.ac.uk}
\address{Warwick Mathematics Institute, Zeeman Building, Coventry CV4 7AL, UK}

\author{Markus Hausmann}
\email{hausmann@math.ku.dk}
\address{Department of Mathematical Sciences, University of Copenhagen, Universitetsparken 5, 2100 K{\o}benhavn {\O}, Denmark}
\thanks{TB was supported by the Danish National Research Foundation Grant DNRF92 and the European Unions Horizon 2020 research and innovation programme under the Marie Sklodowska-Curie grant agreement No. 751794. JPCG is grateful to the EPSRC for support from
  EP/P031080/1. MH was supported by the Danish National Research Foundation Grant DNRF92.\\
All three authors are grateful to the Isaac Newton Institute (funded by EPSRC Grant number EP/R014604/1) for the opportunity to talk during the Homotopy Harnessing Higher Structures Programme.}



\begin{abstract} We study the Balmer spectrum of the category of finite $G$-spectra for a compact Lie group $G$, extending the work for finite $G$ by Strickland, Balmer--Sanders, Barthel--Hausmann--Naumann--Nikolaus--Noel--Stapleton and others. We give a description of the underlying set of the spectrum and show that the Balmer topology is completely determined by the inclusions between the prime ideals and the topology on the space of closed subgroups of $G$. Using this, we obtain a complete description of this topology for all abelian compact Lie groups and consequently a complete classification of thick tensor-ideals. For general compact Lie groups we obtain such a classification away from a finite set of primes $p$.
\end{abstract}
\maketitle

\tableofcontents

\section{Introduction}
Let $\cC$ be a triangulated category equipped with a compatible symmetric monoidal structure. A central problem in studying $\cC$ is the classification of its thick tensor-ideals. This problem can be reformulated~\cite{Bal05} as determining the Balmer spectrum $\Spc(\cC)$, a topological space associated to $\cC$ whose points are the prime ideals of $\cC$.
One of the first examples where a complete description of $\Spc(\cC)$ was obtained is the thick subcategory theorem of Hopkins--Smith \cite{HS98} for the category of finite spectra, which is fundamental to our understanding of the stable homotopy category. In recent years there has been significant progress towards obtaining such a description also for the category of finite genuine $G$-spectra $\mathcal{SH}^c_G$ for a finite group $G$. Balmer--Sanders \cite{BS17b}, building on unpublished work of Strickland, gave a description of the underlying set of $\Spc(\mathcal{SH}^c_G)$ and showed that the topology and hence the classification of thick tensor-ideals is determined by the inclusions between the prime ideals. They further determined which inclusions occur for $G=\Z/p$ and deduced a complete description of the Balmer spectrum for all finite groups $G$ of square-free order. This classification was extended to all finite abelian groups in \cite{BHN^+17}, by computing the height shift of the generalized Tate--construction for Lubin--Tate theory as well as making use of fixed-point properties of explicit finite equivariant complexes previously considered by Arone \cite{Aro98} and Arone--Lesh \cite{AL17}.

The goal of the present paper is to generalize these results to the category $\mathcal{SH}_{G}^c$ of finite genuine $G$-spectra for a compact Lie group $G$. Similarly to the finite group case, we obtain a complete description of $\Spc(\mathcal{SH}_{G}^c)$ as a set for all compact Lie groups, and a classification of the Balmer topology and the tt-ideals for all \emph{abelian} compact Lie groups, in particular tori. While some of the structural properties are similar to the finite group case, there are also new phenomena, 
the most notable being the interplay of the Hausdorff topology on the space of closed subgroups of $G$ (which is discrete for finite groups) and the Balmer topology on $\Spc(\mathcal{SH}_{G}^c)$. In particular, it is no longer true that the Balmer topology is determined by the inclusions among prime ideals alone. Another difference to the finite group case is that even in the $p$-local category $\mathcal{SH}_{G,(p)}^c$ not every thick tensor-ideal is finitely generated.

\subsection*{Statement of the results} We now describe the results of this paper in more detail. Recall that a thick tensor-ideal (=tt-ideal) in a tensor triangulated category $(\cC,\otimes,\mathbbm{1})$ is a triangulated subcategory that is closed under retracts and tensor product with an arbitrary object of $\cC$. A proper tt-ideal $I$ is said to be prime if $X\otimes Y\in I$ implies that $X\in I$ or $Y\in I$. The Balmer spectrum is the set of all prime tt-ideals, with Balmer topology generated by the closed sets $\supp(X)=\{\wp\mid X\notin \wp\}$ for all $X\in \cC$. We are interested in determining this spectrum $\Spc(\cC)$ in the case where $\cC$ is the homotopy category $\mathcal{SH}_{G}^c$ of finite $G$-spectra for a compact Lie group $G$ with the smash product $\wedge$ as the monoidal product. For simplicity, we work in the $p$-local homotopy category $\mathcal{SH}_{G,(p)}^c$, from which the global result can be reconstructed as explained in \cite[Sec. 6]{BS17b}. We rely on the computation of the spectrum of the homotopy category of non-equivariant finite $p$-local spectra, due to Hopkins and Smith \cite{HS98}, which we briefly recall. Let $\NN=\Z_{\ge 0}\cup \{\infty\}$. Then every prime ideal is of the form $P(n)=\{ X\mid  K(n-1)_*(X)=0\}$ for some $n\in \NN_{>0}$, where $K(n)$ denotes the $n$-th Morava $K$-theory. These primes $P(n)$ are pairwise distinct and form a descending chain
\[ 
P(1)\supset P(2) \supset \hdots \supset P(\infty).
\]
In order to construct prime ideals for finite $G$-spectra, one makes use of the geometric fixed point functors
\[ 
\Phi^H\colon\mathcal{SH}_{G,(p)}^c\to \mathcal{SH}_{(p)}^c 
\]
for all closed subgroups $H$ of $G$. Since these functors are symmetric monoidal, the subcategories
\[ 
P_G(H,n)=\{X\in \mathcal{SH}_{G,(p)}^c\mid K(n-1)_*(\Phi^H(X))=0\} 
\]
are again prime ideals. Our first result (\Cref{thm:underlyingset}) then says:

\begin{theorem}[(The underlying set of $\Spc(\mathcal{SH}_{G,(p)}^c)$)] Let $G$ be a compact Lie group. Then every prime ideal of the homotopy category of finite $p$-local $G$-spectra $\mathcal{SH}_{G,(p)}^c$ is of the form $P_G(H,n)$ for some closed subgroup $H$ of $G$ and $n\in \NN_{>0}$. Moreover, $P_G(H,m)=P_G(K,n)$ if and only if $H$ is $G$-conjugate to $K$ and $m=n$.
\end{theorem}
The case of finite groups is due to Balmer--Sanders \cite{BS17b}. Our proof is a bit different and builds on a general method introduced by Hovey--Palmieri--Strickland in \cite{HPS97}.

We then study the Balmer topology on
$\Spc(\mathcal{SH}_{G,(p)}^c)$. For general $G$, our main result says
that it can be expressed in terms of the inclusions between prime
ideals and the Hausdorff topology on the set $\Sub(G)$ of closed
subgroups of $G$. Let $\Sub(G)/G$ denote the orbit space by the
conjugation $G$-action, equipped with the quotient topology making it
a compact and totally-disconnected Hausdorff space \cite[Thm.
5.6.1]{tD79}. Using the poset structure on the set of prime ideals, we define a function
$f\colon\Sub(G)/G\to \NN$ to be  \emph{admissible} if $P_G(K,f(K))$ is not contained in $P_G(H,f(H)+1)$ for all subgroups $K,H$ with $f(H)<\infty$. The Balmer topology can then be described in the following way (\Cref{cor:topbasis}):
\begin{theorem}[(The topology on $\Spc(\mathcal{SH}_{G,(p)}^c)$)] For every compact Lie group $G$ the open sets
\[ 
\{ P_G(H,n)\mid  n\leq f(H) \} 
\]
are a basis for the Balmer topology on $\Spc(\mathcal{SH}_{G,(p)}^c)$, where $f$ ranges through all locally constant admissible functions $\Sub(G)/G\to \NN$.
\end{theorem}

Since admissibility depends only on the poset structure  on the set of prime ideals, the Balmer topology only depends on this poset structure and the Hausdorff topology on $\Sub(G)/G$. Consequently, the classification of tt-ideals can also be described in terms of admissible functions. Given a finite $p$-local $G$-spectrum $X$, let $\type{X}\colon\Sub(G)/G\to \NN$ be the \emph{type function} of $X$, which assigns to a conjugacy class $[H]$ the chromatic type of the geometric fixed points $\Phi^H(X)$, i.e., the minimal $n$ such that $\Phi^H(X)\notin P(n+1)$. This definition extends to tt-ideals $I$, via $\type{I}(H)=\min(\type{X}(H)\mid  X\in I)$. Then we have (\Cref{thm:criterion}):
\begin{theorem}[(Classification of tt-ideals in $\mathcal{SH}_{G,(p)}^c$)]
Let $G$ be a compact Lie group.
\begin{enumerate}
\item The assignment $I\mapsto \type{I}$ defines a bijection from the set of tt-ideals in the homotopy category of finite $p$-local $G$-spectra $\mathcal{SH}_{G,(p)}^c$ to the set of upper semi-continuous admissible functions, i.e., those admissible functions $f\colon\Sub(G)/G\to \NN$ for which the preimage $f^{-1}([0,n))$ is open for all $n\in \NN$.
\item Such a tt-ideal $I$ is finitely generated if and only if its type function $\type{I}$ is locally constant.
\end{enumerate}
\end{theorem}
The last item can be reformulated as saying that given a function $f\colon\Sub(G)/G\to \NN$, there exists a finite $G$-spectrum $X$ with $\type{X}=f$ if and only if $f$ is admissible and locally constant.

In order to obtain an explicit description of the Balmer topology and the classification of tt-ideals, it is hence necessary to understand the poset structure on $\Spc(\mathcal{SH}_{G,(p)}^c)$ and the topology on $\Sub(G)/G$. The latter topology is fairly well-understood thanks to work of Montgomery-Zippin \cite{MZ42} and tom Dieck~\cite[Sec.~5.6]{tD79}. Unravelling the definitions, determining the poset structure boils down to the following basic geometric question: Given two closed subgroups $K$ and $H$ of $G$ and numbers $n,m\in \NN_{>0}$, does there exist a finite $G$-CW complex $X$ such that the $K$-fixed points $X^K$ have type $\geq n$ and the $H$-fixed points~$X^H$ have type $<m$? The answer to this question is `yes' if and only if $P_G(K,n)\not\subseteq P_G(H,m)$.

This problem appears to be difficult in general and is already wide open for general finite groups $G$. For abelian compact Lie groups $A$, however, the answer is as follows (\Cref{thm:abelian}). Here, we write $rk_p(A')$ for the $p$-rank of a finite abelian group $A'$, i.e., the dimension of $A'\otimes \Z/p$ over $\Z/p$.

\begin{theorem}[(Abelian groups)] Let $A$ be an abelian compact Lie group. Then we have an inclusion $P_A(K,n)\subseteq P_A(H,m)$ if and only if $K$ is a subgroup of $H$, $\pi_0(H/K)$ is a $p$-group and
\[ 
m+ rk_p(\pi_0(H/K))\leq n.
\]

Hence, a function $f\colon \Sub(A)\to \NN$ is admissible if and only if 
\[
f(H)+ rk_p(\pi_0(H/K)) \geq f(K)
\]
for all $K\subseteq H$ with $\pi_0(H/K)$ a $p$-group.
\end{theorem}

Note in particular that if $K\subseteq H$ is a subgroup with quotient $H/K$ a torus, then there is an inclusion $P_A(K,n)\subseteq P_A(H,n)$ for every $n$, as follows from a Morava $K$-theory version of the classical Borel Localization Theorem. This is the main new ingredient in addition to the finite abelian case of \cite{BHN^+17}. 

\begin{figure}[h!]
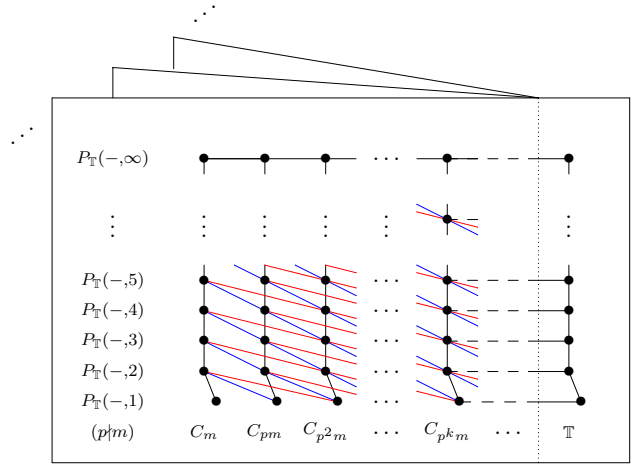

\[
\centering{\scalebox{0.8}{\xy
%
%
(-25,55)*{\iddots};
(5,75)*{\iddots};
%
{\ar@{-} (-20,0)*{};(-20,60)*{}};
{\ar@{-} (-20,60)*{};(60,60)*{}};
{\ar@{..} (60,60)*{};(60,0)*{}};
{\ar@{-} (60,0)*{};(-20,0)*{}};
{\ar@{-} (60,60)*{};(75,60)*{}};
{\ar@{-} (75,60)*{};(75,0)*{}};
{\ar@{-} (75,0)*{};(60,0)*{}};
{\ar@{-} (-10,65)*{};(60,60)*{}};
{\ar@{-} (-10,65)*{};(-10,60)*{}};
{\ar@{-} (0,70)*{};(60,60)*{}};
{\ar@{-} (0,70)*{};(0,64.8)*{}};
%
%
(5,5)*{\scriptstyle C_{m}};
{\ar@{-} (7,10)*{};(5,15)*{}};
{\ar@{-} (5,15)*{};(5,20)*{}};
{\ar@{-} (5,20)*{};(5,25)*{}};
{\ar@{-} (5,25)*{};(5,30)*{}};
{\ar@{-} (5,30)*{};(5,32.5)*{}};
{\ar@{-} (5,47.5)*{};(5,50)*{}};
%
%
{\ar@{-} (17,10)*{};(15,15)*{}};
{\ar@{-} (15,15)*{};(15,20)*{}};
{\ar@{-} (15,20)*{};(15,25)*{}};
{\ar@{-} (15,25)*{};(15,30)*{}};
{\ar@{-} (15,30)*{};(15,32.5)*{}};
{\ar@{-} (15,47.5)*{};(15,50)*{}};
%
%
{\ar@{-} (27,10)*{};(25,15)*{}};
{\ar@{-} (25,15)*{};(25,20)*{}};
{\ar@{-} (25,20)*{};(25,25)*{}};
{\ar@{-} (25,25)*{};(25,30)*{}};
{\ar@{-} (25,30)*{};(25,32.5)*{}};
{\ar@{-} (25,47.5)*{};(25,50)*{}};
%
%
{\ar@{-} (47,10)*{};(45,15)*{}};
{\ar@{-} (45,15)*{};(45,20)*{}};
{\ar@{-} (45,20)*{};(45,25)*{}};
{\ar@{-} (45,25)*{};(45,30)*{}};
{\ar@{-} (45,30)*{};(45,32.5)*{}};
{\ar@{-} (45,40)*{};(45,37.5)*{}};
{\ar@{-} (45,40)*{};(45,42.5)*{}};
{\ar@{-} (45,47.5)*{};(45,50)*{}};
%
%
{\ar@{-} (67,10)*{};(65,15)*{}};
{\ar@{-} (65,15)*{};(65,20)*{}};
{\ar@{-} (65,20)*{};(65,25)*{}};
{\ar@{-} (65,25)*{};(65,30)*{}};
{\ar@{-} (65,30)*{};(65,32.5)*{}};
{\ar@{-} (65,47.5)*{};(65,50)*{}};
%
%
{\ar@{-}@[blue] (17,10)*{};(5,15)*{}};
{\ar@{-}@[blue] (15,15)*{};(5,20)*{}};
{\ar@{-}@[blue] (15,20)*{};(5,25)*{}};
{\ar@{-}@[blue] (15,25)*{};(5,30)*{}};
{\ar@{-}@[blue] (15,30)*{};(10,32.5)*{}};
{\ar@{-} (15,50)*{};(5,50)*{}};
%
%
{\ar@{-}@[blue] (27,10)*{};(15,15)*{}};
{\ar@{-}@[blue] (25,15)*{};(15,20)*{}};
{\ar@{-}@[blue] (25,20)*{};(15,25)*{}};
{\ar@{-}@[blue] (25,25)*{};(15,30)*{}};
{\ar@{-}@[blue] (25,30)*{};(20,32.5)*{}};
{\ar@{-} (25,50)*{};(5,50)*{}};
%
%
{\ar@{-}@[blue] (30,12.5)*{};(25,15)*{}};
{\ar@{-}@[blue] (30,17.5)*{};(25,20)*{}};
{\ar@{-}@[blue] (30,22.5)*{};(25,25)*{}};
{\ar@{-}@[blue] (30,27.5)*{};(25,30)*{}};
{\ar@{-} (30,50)*{};(25,50)*{}};
%
%
{\ar@{-}@[blue] (47,10)*{};(40,13.5)*{}};
{\ar@{-}@[blue] (45,15)*{};(40,17.5)*{}};
{\ar@{-}@[blue] (45,20)*{};(40,22.5)*{}};
{\ar@{-}@[blue] (45,25)*{};(40,27.5)*{}};
{\ar@{-}@[blue] (45,30)*{};(40,32.5)*{}};
{\ar@{-}@[blue] (45,40)*{};(40,42.5)*{}};
{\ar@{-} (45,50)*{};(40,50)*{}};
%
%
{\ar@{-}@[blue] (50,12.5)*{};(45,15)*{}};
{\ar@{-}@[blue] (50,17.5)*{};(45,20)*{}};
{\ar@{-}@[blue] (50,22.5)*{};(45,25)*{}};
{\ar@{-}@[blue] (50,27.5)*{};(45,30)*{}};
{\ar@{-}@[blue] (50,37.5)*{};(45,40)*{}};
{\ar@{--}@[purple] (50,40)*{};(45,40)*{}};
{\ar@{-} (50,50)*{};(45,50)*{}};
%
%
{\ar@{-}@[red] (27,10)*{};(5,15)*{}};
{\ar@{-}@[red] (25,15)*{};(5,20)*{}};
{\ar@{-}@[red] (25,20)*{};(5,25)*{}};
{\ar@{-}@[red] (25,25)*{};(5,30)*{}};
{\ar@{-}@[red] (25,30)*{};(15,32.5)*{}};
%
%
{\ar@{-}@[red] (30,11.25)*{};(15,15)*{}};
{\ar@{-}@[red] (30,16.25)*{};(15,20)*{}};
{\ar@{-}@[red] (30,21.25)*{};(15,25)*{}};
{\ar@{-}@[red] (30,26.25)*{};(15,30)*{}};
{\ar@{-}@[red] (30,31.25)*{};(25,32.5)*{}};
%
%
{\ar@{-}@[red] (30,13.75)*{};(25,15)*{}};
{\ar@{-}@[red] (30,18.75)*{};(25,20)*{}};
{\ar@{-}@[red] (30,23.75)*{};(25,25)*{}};
{\ar@{-}@[red] (30,28.75)*{};(25,30)*{}};
%
%
{\ar@{-}@[red] (47,10)*{};(40,11.75)*{}};
{\ar@{-}@[red](45,15)*{};(40,16.25)*{}};
{\ar@{-}@[red] (45,20)*{};(40,21.25)*{}};
{\ar@{-}@[red] (45,25)*{};(40,26.25)*{}};
{\ar@{-}@[red] (45,30)*{};(40,31.25)*{}};
{\ar@{-}@[red] (45,40)*{};(40,41.25)*{}};
%
%
{\ar@{-}@[red] (50,13.75)*{};(45,15)*{}};
{\ar@{-}@[red] (50,18.75)*{};(45,20)*{}};
{\ar@{-}@[red] (50,23.75)*{};(45,25)*{}};
{\ar@{-}@[red] (50,28.75)*{};(45,30)*{}};
{\ar@{-}@[red] (50,38.75)*{};(45,40)*{}};
%
{\ar@{-}@[gray] (67,10)*{};(60,10)*{}};
{\ar@{-}@[gray] (65,15)*{};(60,15)*{}};
{\ar@{-}@[gray] (65,20)*{};(60,20)*{}};
{\ar@{-}@[gray] (65,25)*{};(60,25)*{}};
{\ar@{-}@[gray] (65,30)*{};(60,30)*{}};
{\ar@{-}@[gray] (65,50)*{};(60,50)*{}};
{\ar@{--}@[gray] (60,10)*{};(47,10)*{}};
{\ar@{--}@[gray] (60,15)*{};(45,15)*{}};
{\ar@{--}@[gray] (60,20)*{};(45,20)*{}};
{\ar@{--}@[gray] (60,25)*{};(45,25)*{}};
{\ar@{--}@[gray] (60,30)*{};(45,30)*{}};
{\ar@{--}@[gray] (60,50)*{};(45,50)*{}};
%
%
(7,10)*{\bullet};
(5,15)*{\bullet};
(5,20)*{\bullet};
(5,25)*{\bullet};
(5,30)*{\bullet};
(5,40)*{\vdots};
(5,50)*{\bullet};
%
%
(15,5)*{\scriptstyle C_{pm}};
(17,10)*{\bullet};
(15,15)*{\bullet};
(15,20)*{\bullet};
(15,25)*{\bullet};
(15,30)*{\bullet};
(15,40)*{\vdots};
(15,50)*{\bullet};
%
%
(25,5)*{\scriptstyle C_{p^2m}};
(27,10)*{\bullet};
(25,15)*{\bullet};
(25,20)*{\bullet};
(25,25)*{\bullet};
(25,30)*{\bullet};
(25,40)*{\vdots};
(25,50)*{\bullet};
%
%
(35,10)*{\hdots};
(35,15)*{\hdots};
(35,20)*{\hdots};
(35,25)*{\hdots};
(35,30)*{\hdots};
(35,40)*{\vdots};
(35,50)*{\hdots};
%
%
(35,5)*{\hdots};
%
%
(45,5)*{\scriptstyle C_{p^km}};
(47,10)*{\bullet};
(45,15)*{\bullet};
(45,20)*{\bullet};
(45,25)*{\bullet};
(45,30)*{\bullet};
(45,40)*{\bullet};
(45,50)*{\bullet};
%
%
(55,5)*{\hdots};
%
%
(65,5)*{\scriptstyle \T};
(67,10)*{\bullet};
(65,15)*{\bullet};
(65,20)*{\bullet};
(65,25)*{\bullet};
(65,30)*{\bullet};
(65,40)*{\vdots};
(65,50)*{\bullet};
%
%
(-10,10)*{\scriptstyle P_{\T}(-,1)};
(-10,15)*{\scriptstyle P_{\T}(-,2)};
(-10,20)*{\scriptstyle P_{\T}(-,3)};
(-10,25)*{\scriptstyle P_{\T}(-,4)};
(-10,30)*{\scriptstyle P_{\T}(-,5)};
(-10,40)*{\vdots};
(-10,50)*{\scriptstyle P_{\T}(-,\infty)};
%
(-10,5)*{\scriptstyle (p\nmid m)};
\endxy}}
\]
\caption{An illustration of the $p$-local Balmer spectrum for the circle group~$\T$. We refer to \Cref{fig:spccircle} for an interpretation of this picture.}
\label{fig:spccircleintro}
\end{figure}

For general compact Lie groups $G$, we get a complete answer for all but finitely many $p$, namely those $p$ appearing as divisors of the order of Weyl groups $|W_GH|$, where $H$ ranges through all closed subgroups with finite Weyl group in $G$ (that this number is indeed finite is a theorem of tom Dieck \cite[Thm.~1]{tD77}). This is described in \Cref{thm:order}. Finally, if one is willing to study the question simultaneously for all compact Lie groups, it is possible to reduce it to the case where $G$ is an extension of a finite $p$-group by an abelian compact Lie group, see \Cref{sec:reduce}.

\begin{remark} The Balmer spectrum $\Spc(\mathcal{SH}_{G,\Q}^c)$ of
  the homotopy category of finite rational $G$-spectra was determined for
  all compact Lie groups $G$ in \cite{Gre17},  but this also follows from 
  \Cref{thm:inclusions} (poset structure) and
  \Cref{cor:topbasis} (topology).  
\end{remark}

\subsection*{Acknowledgements}
We would like to thank the two anonymous referees for helpful comments on an earlier version of this paper.

\section{Recollections}
Let us begin by quickly recalling the proof of the classical thick subcategory theorem as well as collecting some basic notions from stable equivariant homotopy theory that will be used throughout the paper. 

As explained in Section 6 of \cite{BS17b} for finite groups, the computation of $\Spc(\SH_{G}^c)$ for a compact Lie group $G$ may be separated into that of $\Spc(\SH_{G,(p)}^c)$ for all primes $p$. Throughout the paper, we will therefore work in the $p$-local homotopy category for a fixed prime $p$, and all our constructions are implicitly $p$-local.

\subsection{The classical thick subcategory theorem}\label{ssec:classicaltst}

As the prototypical example, we quickly recall the key steps in the computation of the Balmer spectrum of the homotopy category of finite non-equivariant $p$-local spectra $\mathcal{SH}^c_{(p)}$, following Hopkins and Smith \cite{HS98}. 

For $n\in \N_{>0}$, let $K(n)$ denote the $n$-th Morava $K$-theory spectrum at $p$, with coefficients $K(n)_*=\mathbb{F}_p[v_n^{\pm 1}]$ and $|v_n|=2(p^n-1)$. Furthermore, $K(0)$ is defined to be the rational Eilenberg--Mac Lane spectrum $H\Q$.   For $n>0$ we write
\[ 
P(n)=\{ X\in \mathcal{SH}^c_{(p)}\mid  K(n-1)_*(X)=0\}
\]
and $P(\infty)=\bigcap_{n\in \N}P(n)$ for the set of all contractible finite $p$-local spectra. Every $P(n)$ is a prime ideal, since the $K(n)$ satisfy a K{\"u}nneth formula. Moreover, note that $\mathcal{SH}^c_{(p)}$ is generated by the unit of the smash product, the $p$-local sphere spectrum, so any thick subcategory is automatically a tensor ideal. The thick subcategory theorem~\cite[Thm.~7]{HS98} for $\mathcal{SH}^c_{(p)}$ provides a complete computation of the Balmer spectrum $\Spc(\mathcal{SH}^c_{(p)})$:

\begin{theorem}[Hopkins--Smith]\label{thm:classicaltst}
If $\cC \subseteq \mathcal{SH}^c_{(p)}$ is a proper thick subcategory, then there exists a unique $n \in \NN_{>0}$ such that $\cC = P(n)$. Moreover, there are proper inclusions 
\[ 
P(1)\supset P(2) \supset \hdots \supset P(\infty)=\{0\}. 
\]
\end{theorem}

The proof can be divided into three steps. 

\begin{enumerate}
	\item Firstly, Ravenel \cite[Thm.~2.11]{Rav84} showed that, if $n>0$ and $X$ is a finite spectrum for which $K(n)_*(X)=0$, then also $K(n-1)_*(X)=0$. This translates into the inclusion $P(n+1) \subseteq P(n)$. Given a finite spectrum $X$, there hence exists a maximal $n$ such that $X\in P(n)$. This $n$ is called the \emph{type of $X$}. If $X\notin P(1)$, or in other words if the rational homology of $X$ is non-trivial, $X$ is said to be of type $0$. 
	\item Secondly, these inclusions are all proper, as was first shown by Mitchell \cite{Mit85}, by constructing a finite spectrum $F(n)$ with $K(n)_*(F(n)) \neq 0$ and $K(n-1)_*(F(n)) = 0$.
	\item Thirdly, as an application of the nilpotence theorem, Hopkins and Smith showed that if another finite spectrum $Y$ has type larger than or equal to the type of $X$, then $Y$ already lies in the tt-ideal $\tti{X}$ generated by $X$. As a consequence, any finite spectrum of type $n$ generates $P(n)$. It follows that the $P(n)$ make up all the proper thick tensor ideals and in particular all the prime ideals in $\mathcal{SH}^c_{(p)}$. Moreover, the topology on $\Spc(\mathcal{SH}^c_{(p)})$ is determined by the closure of points, which are given by $\overline{\{P(n)\}} =  \{P(n),P(n+1),\hdots,P(\infty) \}$.
\end{enumerate}

These steps are also manifest in the computation of the Balmer spectrum of the stable equivariant homotopy category. The goal of the next section is to establish an analogue of (3) for all compact Lie groups $G$ as well as an equivariant generalization of the nilpotence theorem. In contrast, Steps (1) and (2) turn out to be more subtle in the equivariant context, and so far only partial generalizations are known; this will be the subject of the remainder of the paper.

\subsection{Stable equivariant homotopy theory} \label{sec:eqhomotopy}
We recall some notions and results from equivariant stable homotopy theory for a compact Lie group $G$ that we will need throughout the paper. Classical references in the Lewis--May approach are \cite{LMSM86, GM95a,may_equivariant96}, a treatment of $G$-orthogonal spectra is given in \cite{MM02}, and \cite[Secs. 5 and 6]{MNN17} contains a discussion of some of the constructions in an $\infty$-categorical setting. 

We denote by $\SH_{G,(p)}$ the homotopy category of $p$-local $G$-spectra for a compact Lie group~$G$ indexed on a complete universe and write $\Sub(G)$ for the set of closed subgroups of $G$. For $H \in \Sub(G)$, the suspension spectrum $G/H_+ = \Sigma^{\infty}G/H_+$ is a $G$-spectrum, and the $H$-homotopy groups of $X \in \SH_{G,(p)}$ are defined as
\[
\pi_*^H(X) = [G/H_+,X]_*^G,
\]
i.e., the graded maps in $\SH_{G,(p)}$ from the suspension spectrum of $G/H_+$ into $X$.
A map of $G$-spectra is an equivalence if and only if it induces an isomorphism on $\pi_*^H$ for all closed subgroups~$H$ of $G$.

The category $\SH_{G,(p)} = (\SH_{G,(p)},\wedge,\mathbb{S}_G)$ is a tensor triangulated category which is compactly generated by the set of dualizable objects $\{G/H_+\}_{H \in \Sub(G)}$ and with compact unit given by the $G$-equivariant sphere spectrum $\mathbb{S}_G=G/G_+$. For every real $G$-representation $V$, the suspension spectrum of the one-point compactification $S^V$ is an invertible object in $\SH_{G,(p)}$. As usual, we denote the full subcategory of $\SH_{G,(p)}$ spanned by the compact objects by $\SH_{G,(p)}^c$. The objects of $\SH_{G,(p)}^c$, which we call the \emph{finite $G$-spectra}, can be described in the following way:

\begin{lemma} \label{lem:finiteg}
Every finite $p$-local $G$-spectrum $X$ is of the form $S^{-V}\wedge \Sigma^{\infty}A$, where $V$ is a $G$-representation and $A$ is a homotopy retract of the $p$-localization of a based finite $G$-CW complex.
\end{lemma}
Here, to avoid subleties with $p$-localization of $G$-spaces, the finite $G$-CW complex can always be chosen to be a double suspension, so all fixed points are simply-connected and the $p$-localization is given by the mapping telescope over repeated multiplication by $q$ for all primes $q\neq p$, using the comultiplication from the suspension coordinates. The lemma is (a $p$-local variation of) a classical fact about $G$-spectra and follows from every $G$-spectrum being a filtered homotopy colimit over $G$-spectra of the form $S^{-V}\wedge \Sigma^{\infty} (-)$, see \cite[Prop. I.4.7]{LMSM86}.

\subsubsection{Pullback along a group homomorphism}
Let $\varphi\colon G\to K$ be a continuous group homomorphism. Then every $K$-space $A$ gives rise to a $G$-space $\varphi^*(A)$ by taking the same underlying space and pulling back the $K$-action to a $G$-action along $\varphi$. There is a spectrum analog $\varphi^*\colon\SH_{K,(p)}\to \SH_{G,(p)}$, technically obtained as a composite of a space level pullback functor and a change of universe functor, see \cite[Secs.~II.1--4]{LMSM86} or \cite[Secs.~V.1--3]{MM02}. It is uniquely determined by the fact that it is exact, symmetric monoidal, preserves all coproducts and sends a suspension spectrum $\Sigma^{\infty} A$ to the suspension spectrum $\Sigma^{\infty}\varphi^*(A)$. In particular, $\varphi^*$ preserves dualizable objects and hence restricts to a functor $\varphi^*\colon\SH^c_{K,(p)}\to \SH^c_{G,(p)}$.

We need two special cases of this construction, called restriction and inflation:

\subsubsection{Restriction and induction}\label{sec:restriction}
In the case where $\varphi$ is the inclusion of a closed subgroup $H$ into $G$, the pullback functor is called \emph{restriction} and denoted $\res_H^G$. In that case it has a left adjoint called \emph{induction} that is denoted $\Ind_H^G$. When $A$ is a based $H$-space, there is a natural equivalence $\Ind_H^G(\Sigma^{\infty}A)\simeq \Sigma^{\infty} (G_+\wedge _H A)$, which follows from the analogous property of $\res_H^G$. Furthermore, for a pair of an $H$-spectrum $X$ and a $G$-spectrum $Y$, Frobenius reciprocity provides an equivalence
\[ 
\Ind_H^G(X\wedge \res_H^G(Y))\simeq \Ind_H^G(X)\wedge Y. 
\]
One way to see this is to note that adjunction gives a natural map from the left to the right, and it is enough to check that this map is an equivalence in the case where $X$ and $Y$ vary through the compact generators. Since these are suspension spectra and $\Ind_H^G, \res_H^G$, as well as~$\wedge$ commute with $\Sigma^{\infty}$, this reduces to the space level analog which is easily checked.

\subsubsection{Inflation} \label{sec:inflation}
When $\varphi$ is the projection $p\colon G\to G/N$ for a closed normal subgroup of $G$, the pullback functor is called \emph{inflation} and denoted $i_N^*$. Most important for us is the case $N=G$, in which case the inflation functor turns an ordinary spectrum into a $G$-spectrum with `trivial action'.

\subsubsection{Geometric fixed points} \label{sec:geometricfixed}  For every closed subgroup $H$ of $G$, there is a \emph{geometric fixed point functor }
\[ 
\Phi^H = \Phi_G^H \colon \SH_{G,(p)}\to \SH_{(p)},
\]
obtained by composing the restriction $\res_H^G\colon \SH_{G,(p)} \to \SH_{H,(p)}$ with the (absolute) geometric fixed point functor for $H$, see \cite[Sec.~II.9]{LMSM86}, \cite[Sec.~V.4]{MM02} or \cite[Sec.~6.2]{MNN17}. This functor is exact, symmetric monoidal, preserves all coproducts, and satisfies $\Phi^H(\Sigma^{\infty}A)\simeq \Sigma^{\infty}(A^H)$ for all based $G$-CW complexes $A$. Since the dualizable objects and the compact objects coincide in $\SH_{G,(p)}$, the geometric fixed point functor $\Phi^H$ preserves compact objects and hence restricts to a functor $\Phi^H\colon \SH_{G,(p)}^c \to \SH_{(p)}^c$. Every element $g\in G$ induces conjugation equivalences $\Phi^H\simeq \Phi^{H^g}$, where $H^g=g^{-1}Hg$. Hence, up to equivalence, the geometric fixed point functors only depend on the conjugacy class of $H$. Moreover, the collection $\{\Phi^H\}_{H \in \Sub(G)}$ is jointly conservative, i.e., a map $f$ in $\SH_{G,(p)}$ is an equivalence if and only if $\Phi^H(f)$ is an equivalence in $\SH_{(p)}$ for all $H\in\Sub(G)$ (see \cite[Prop. 3.3.10]{Sch18} for a proof). 

Finally, geometric fixed points have the following behaviour with respect to pullback along a continuous group homomorphism $\varphi\colon G\to K$. For a closed subgroup $H$ of $G$ there are natural equivalences
\begin{equation} \label{eq:inflation}
\Phi^H(\varphi^*(X))\simeq \Phi^{\varphi(H)}(X).
\end{equation}
These can be derived from their space-level analogs using that both sides are exact, symmetric monoidal, coproduct preserving functors in $X$.

\subsubsection{Borel completion}\label{sec:borelcompletion}
A $G$-spectrum $Y$ is called \emph{Borel-complete} if the map $Y\to F(EG_+,Y)$ is an equivalence, see \cite[Sec. 6.3]{MNN17} for a recent reference. Here, $F(-,-)$ denotes the function $G$-spectrum between two $G$-spectra. If the underlying non-equivariant spectrum of a Borel-complete $G$-spectrum $Y$ is contractible, then $Y$ is already contractible as a $G$-spectrum, which follows from an induction over the cells of $EG_+$. For every spectrum $Z$ equipped with a $G$-action in the naive sense (i.e., a module over the spherical group ring $\Sigma^{\infty} G_+$, or in $\infty$-categorical language a functor from $BG$ to the $\infty$-category of spectra) there exists a \emph{Borel spectrum} $\underline{Z}\in \mathcal{SH}_{G,(p)}$ uniquely determined by the property that it is Borel complete and that its underlying spectrum with naive $G$-action is equivalent to $Z$.

\section{Equivariant support, equivariant prime ideals and the equivariant nilpotence theorem}

In this section, we use the geometric fixed point functors and the non-equivariant nilpotence theorem to construct a support theory that allows us to detect inclusions between equivariant tt-ideals, thereby providing an equivariant generalization of Step (3) in the proof of \Cref{thm:classicaltst}. We then use these results to determine the underlying set of the Balmer spectrum of the category of finite $p$-local $G$-spectra. Although not needed in the remainder of the paper, we also prove an equivariant version of the nilpotence theorem for all compact Lie groups.

\subsection{An abstract thick tensor ideal theorem}

Let $\cC = (\cC,\otimes,1)$ be a closed symmetric monoidal triangulated category with arbitrary coproducts and compact unit $1$, which implies that every dualizable object in $\cC$ is compact. Following Hovey, Palmieri, and Strickland \cite{HPS97}, we say that $\cC$ is a unital algebraic stable homotopy category if it is compactly generated by a set of dualizable objects. 

Suppose given an indexing set $I$, a collection of graded abelian categories $(\cD_i)_{i \in I}$ with zero objects $0 \in \cD_i$ and a collection of stable homological functors $(\kappa_i\colon \cC \to \cD_i)_{i \in I}$ satisfying the following two conditions:
\begin{enumerate}
	\item Let $i \in I$ and $X,Y \in \cC$. Then  $\kappa_i(X\otimes Y) = 0$ if and only if $\kappa_i(X) = 0$ or $\kappa_i(Y)=0$. 
	\item If $R \in \Alg(\cC)$ is a monoid in $\cC$ with $\kappa_i(R) = 0$ for all $i \in I$, then $R = 0$. 
\end{enumerate}
We will refer to such a collection of functors as a \emph{coarse support theory} for $\cC$. The corresponding notion of \emph{support} is then defined for any $X \in \cC$ as 
\[
\supp(X) = \{i \in I\mid \kappa_i(X) \neq 0\}.
\]
This extends to any collection $\cX \subseteq \cC$ of objects in $\cC$ by setting $\supp(\cX) = \bigcup_{X \in \cX}\supp(X)$. Since the functors $\kappa_i$ are stable homological to graded abelian categories $\cD_i$, then $\supp(X) = \supp(\tti{X})$. 

\begin{remark}
Conditions (1) and (2) translate to the following two properties of the corresponding notion of support:
\begin{enumerate}
	\item[(1')] For any two objects $X,Y \in \cC$, we have $\supp(X \otimes Y) = \supp(X) \cap \supp(Y)$. 
	\item[(2')] A monoid $R \in \cC$ is trivial if and only if $\supp(R) = \varnothing$.
\end{enumerate}
A coarse support theory thus satisfies part of the defining conditions of a support datum as axiomatized by Balmer in \cite[Def.~3.1]{Bal05}, thereby justifying our choice of terminology.
\end{remark}

The next result is an abstract version of the support-theoretic part of the thick subcategory theorem; the proof is a modification of the argument in \cite[Pf.~of~Thm.~5.2.2]{HPS97} and thus essentially due to Hovey, Palmieri, and Strickland. 

\begin{prop}\label{prop:abstractsupport}
Suppose $(\kappa_i\colon \cC \to \cD_i)_{i\in I}$ is a coarse support theory on a unital algebraic stable homotopy category $\cC$ and let $X \in \cC^{c}$ be a compact object and $\cD \subseteq \cC^{c}$ a thick tensor ideal. If $\supp(X) \subseteq \supp(\cD)$, then $X \in \cD$. 
\end{prop}
\begin{proof}
For a full subcategory $\cS \subseteq \cC$, let $\Loc^{\otimes}(\cS)$ be the smallest localizing ideal of $\cC$ containing $\cS$, i.e., the smallest triangulated subcategory of $\cC$ containing $\cS$ which is closed under tensor products with objects of $\cC$. The natural inclusion $\cD \to \cC$ then gives rise to a (Verdier) localization sequence
\[
\Loc^{\otimes}(\cD) \to \cC \to \cC/\Loc^{\otimes}(\cD).
\]
By \cite[Lem.~3.1.6(e) and Thm.~3.3.3]{HPS97}, the functor $\cC \to \cC/\Loc^{\otimes}(\cD)$ admits a right adjoint. The corresponding composite $L_{\cD}^f \colon \cC \to \cC/\Loc^{\otimes}(\cD) \to \cC$ is then a smashing localization functor on $\cC$, so there is a canonical isomorphism $L_{\cD}^fY  = Y \otimes L_{\cD}^f1$ for all $Y \in \cC$. Suppose $X$ satisfies the assumptions of the proposition, then the conclusion is equivalent to $L_{\cD}^fX = 0$, since $\Loc^{\otimes}(\cD) \cap \cC^c = \cD$ by Thomason's Theorem in the form of \cite[Thm.~7.1(ii)]{Gre17}.

Note that we may assume without loss of generality that $X \in \Alg(\cC)$: Indeed, since $X$ is a retract of $X \otimes DX \otimes X$ by dualizability, $X$ and the monoid $X \otimes DX$ have the same support. Moreover, the same observation shows that $X \in \cD$ if and only if $X \otimes DX \in \cD$.  

We now claim that $\kappa_i(L_{\cD}^fX) = 0$ for all $i \in I$. To this end, let $i \in I$ and assume that $\kappa_i(X) = 0$. It follows that $\kappa_i(L_{\cD}^f(X)) = \kappa_i(X \otimes L_{\cD}^f1) = 0$. If $i\in I$ is instead such that $\kappa_i(X) \neq 0$, then $i \in \supp(X) \subseteq \supp(\cD)$, so there exists $Y \in \cD$ with $\kappa_i(Y) \neq 0$. But $L_{\cD}^fY = 0$, hence $\kappa_i((X\otimes L_{\cD}^f1) \otimes Y) = 0$. Therefore $\kappa_i(X \otimes L_{\cD}^f1) = 0$ in this case as well. Consequently, $L_{\cD}^fX$ is a monoid in $\cC$ with $\kappa_i(L_{\cD}^fX) = 0$ for all $i\in I$, thus $L_{\cD}^fX = 0$. 
\end{proof}

We note that the proposition in particular implies that every thick tensor ideal $\cD$ in $\cC^c$ is determined by its support: It consists of all $X\in \cC^c$ such that $\supp(X)\subseteq \supp(\cD)$.

\begin{example}\label{ex:moravak}
Let $\SH_{(p)}$ be the $p$-local stable homotopy category at a prime $p$. For any $n \in \NN=\{0,1,\hdots\}\cup\{\infty\}$, write $K(n)$ for the $p$-local height $n$ Morava $K$-theory spectrum, where we set $K(\infty) = H\F_p$. The nilpotence theorem of Devinatz, Hopkins, and Smith \cite{DHS88,HS98} implies that the collection of homological and symmetric monoidal functors 
\[
\{K(n)_*\colon \SH_{(p)} \to \Mod_{K(n)_*}\}_{n \in \NN}
\]
forms a coarse support theory for $\SH_{(p)}$. In this case, \Cref{prop:abstractsupport} recovers the thick subcategory theorem \cite[Thm.~7]{HS98} of Hopkins and Smith, as explained in \Cref{ssec:classicaltst}.
\end{example}

\subsection{Equivariant support and type functions of finite $G$-spectra}

In order to apply the abstract result of the previous subsection to equivariant stable homotopy theory for a compact Lie group $G$, we need to construct an appropriate coarse support theory. Fix a prime $p$ and write $\SH_{G,(p)}$ for the category of $p$-local $G$-spectra. Recall from \Cref{ex:moravak} that for each $n \in \NN$ there is a Morava $K$-theory spectrum $K(n) \in \SH_{(p)}$ of height $n$.

\begin{definition}
Let $I = \Sub(G) \times \NN$ and define for each $(H,n) \in I$ functors 
\[
\kappa_{(H,n)} = K(n)_*(\Phi^H(-))\colon \SH_{G,(p)} \to \Mod_{K(n)_*}.
\]
As in the previous subsection, we shall write $\supp(X)$ for the support of $X \in \SH_{G,(p)}$ corresponding to the collection of functors $\{\kappa_{(H,n)}\}_{(H,n) \in I}$.
\end{definition}

As before, this notion of support extends to any collection $\cX \subseteq \SH_{G,(p)}$. Since each $\Phi^H$ is exact as well as symmetric monoidal and $K(n)$ has a K{\"u}nneth isomorphism, we have $\supp(X) = \supp(\tti{X})$ for all $X \in \SH_{G, (p)}$. Then we have the following result, which was also obtained in unpublished work of Strickland:

\begin{prop}\label{prop:eqsupport}
The collection $\{\kappa_{(H,n)}\}_{ (H,n)\in \Sub(G) \times \NN}$ is a coarse support theory for $\SH_{G, (p)}$.
\end{prop}
\begin{proof}
We have to verify Conditions (1) and (2) above. For any $(H,n) \in \Sub(G) \times \NN$ and all $X,Y \in \SH_{G, (p)}$, there is a canonical isomorphism
\[
K(n)_*(\Phi^H(X)) \otimes_{K(n)_*} K(n)_*(\Phi^H(Y)) \cong K(n)_*(\Phi^H(X \wedge Y)).
\]
Since $K(n)_*$ is a graded field, this implies that the first condition is satisfied. To see that $\supp$ detects ring objects, suppose that $R \in \Alg(\SH_{G, (p)})$ with $\supp(R) = \varnothing$. By the nilpotence theorem in the form of \cite[Thm.~3~i)]{HS98}, this implies that $\Phi^H(R) = 0$ for all $H \in \Sub(G)$, hence $R = 0$ because $\{\Phi^H\}_{H \in \Sub(G)}$ is jointly conservative.  
\end{proof}

For any nonzero finite ($p$-local) spectrum $F$ there exists an $n\in \N$ such that $K(n)_*(F) \neq 0$, which can be seen using the Atiyah--Hirzebruch spectral sequence.

\begin{definition}
If $X$ is a finite $G$-spectrum $X$, we define its \emph{type function} as the function $\type{X}\colon \Sub(G) \to \NN$ defined by
\[
\type{X}(H) = 
	\begin{cases}
	\min\{n \in \N\mid  K(n)_*(\Phi^H(X)) \neq 0\} & \text{if } \Phi^H(X) \neq 0 \\
	\infty & \text{otherwise}.
	\end{cases}
\]
\end{definition}
It turns out that the central problem for determining the topology on $\Spc(\mathcal{SH}^c_{G,(p)})$ and classifying the tt-ideals is the question of which functions can occur as $\type{X}$ for some $X\in \mathcal{SH}^c_{G,(p)}$.

By a theorem due to Ravenel \cite[Thm.~2.11]{Rav84}, if $F \in \SH_{(p)}^c$ has $K(n)_*(F) = 0$, then $K(n-1)_*(F)$ as well, so for finite $G$-spectra $X$, the function $\type{X}$ contains the same information as $\supp(X)$. As before, we may extend $\type{}$ to subcategories $\cX \subseteq \SH_{G, (p)}^c$ by defining 
\[
\type{\cX}(H) = \min\{\type{X}(H)\mid  X \in \cX\}.
\]
We write $f \ge g$ for two functions $f,g\colon \Sub(G) \to \NN$ if $f(H) \ge g(H)$ for all $H \in \Sub(G)$. The next corollary is now an immediate consequence of \Cref{prop:abstractsupport} and \Cref{prop:eqsupport}:

\begin{cor}\label{cor:eqsupport}
Let $X\in \SH_{G, (p)}^c$ be a finite $G$-spectrum and $I\subseteq \SH_{G, (p)}^c$ a tt-ideal such that $\type{X} \ge \type{I}$. Then $X \in I$.

In particular, the type function $\type{I}$ determines the tt-ideal $I$: It is given by all $X\in \SH_{G, (p)}^c$ such that $\type{X} \ge \type{I}$.
\end{cor}
When $I$ is given by $\tti{Y}$ for another finite $G$-spectrum $Y$, the corollary implies that $X\in \tti{Y}$ if and only if $\type{X}\geq \type{Y}$. This statement is perhaps slightly surprising, because it means that $X$ lies in $\tti{Y}$  if and only if $\Phi^H(X)$ lies in $\tti{\Phi^H(Y)}$  for all closed subgroups $H$, with no condition on how the various geometric fixed points are glued together.

\begin{remark}
Bousfield has extended Ravenel's theorem to all $p$-torsion suspension spectra, see \cite{bousfield_type99}, hence the notion of type is well-defined for this larger class of spectra. We are not aware of an analogue of the classification of thick subcategories in this setting. 
\end{remark}

We end this subsection with a result concerning the type function of induced $G$-spectra, which we shall make repeated use of. Recall from \Cref{sec:restriction} that for a closed subgroup $H$ of a compact Lie group $G$, the restriction functor
\[ 
\res_H^G\colon \mathcal{SH}^c_{G,(p)}\to \mathcal{SH}^c_{H,(p)} 
\]
has a left adjoint denoted $\Ind_H^G$.

\begin{lemma} \label{lem:induced} Let $H$ be a closed subgroup of a compact Lie group $G$, and $X$ a finite $H$-spectrum. Then the following hold:
\begin{enumerate}[i)]
	\item The type of $\Phi^H(\Ind_H^G(X))$ is equal to the type of $\Phi^H(X)$.
	\item If $K$ is another closed subgroup of $G$, then the type of $\Phi^K(\Ind_H^G(X))$ is at least the minimum of the types of $\Phi^{K'}(X)$ for all subgroups $K'$ of $H$ that are conjugate to $K$ in $G$.
	\item If $K$ is not subconjugate to $H$, then the geometric fixed points $\Phi^K(\Ind_H^G(X))$ are trivial.
\end{enumerate}
\end{lemma}
A proof of this result is given in the appendix.

\begin{remark} \label{rem:counter} When $G$ is a finite group, there is in fact an equality in Part $ii)$ above, i.e., in that case the type of $\Phi^K(\Ind_H^G(X))$ equals the minimum of the types of $\Phi^{K'}(\Ind_H^G(X))$ for all subgroups $K'$ of $H$ that are conjugate to $K$ in $G$, see \Cref{rem:finequal} in the appendix.

This is not generally true for compact Lie groups. Note that when $K$ is the trivial group, an equality in Part $ii)$ would mean that the type of $\Phi^{e}\Ind_H^G(X)$ is always the same as the type of $\Phi^e(X)$. Counterexamples to such an equality are given by the following complexes, which play a central role in \cite{BHN^+17} and were studied extensively by Arone, Lesh and Mahowald \cite{Aro98,AM99,AL17}. We refer to \cite[Sec. 4]{BHN^+17} for more details. Let $n>0$ and ${\mathcal P}_{p^n}^\diamond$ denote the unreduced suspension of the complex of proper non-trivial partitions of the set $\{1,\hdots,p^n\}$, equipped with the action of $\Sigma_{p^n}$ by permuting the partitions. Further, let $\overline{\rho}$ denote the reduced regular representation of $\Sigma_{p^n}$, and define $X\in \mathcal{SH}^c_{\Sigma_{p^n},(p)} $ by
\[ X=F(\Sigma^{\infty} {\mathcal P}_{p^n}^\diamond,\mathbb{S}_{\Sigma_{p^n}})\wedge S^{\overline{\rho}}. \]
The reduced regular representation (this time over $\mathbb{C}$) defines an embedding $\Sigma_{p^n}\hookrightarrow U(p^n-1)$, so we can consider $\Ind_{\Sigma_{p^n}}^{U(p^n-1)}X$. Since $\mathcal P_{p^n}$ is a non-trivial wedge of spheres (see \cite[4.109]{OT92}), the underlying non-equivariant spectrum of $X$, i.e., $\Phi^e(X)$, has type $0$. On the other hand, by work of Arone-Mahowald \cite{AM99} and Arone \cite{Aro98}, the underlying non-equivariant spectrum of $\Ind_{\Sigma_{p^{n}}}^{U(p^n-1)}X$ has type $n$, see also the proof of \cite[Thm. 2.2]{BHN^+17}.
\end{remark}

\subsection{The underlying set of $\Spc(\mathcal{SH}^c_{G,(p)})$}
In this subsection we describe the set of prime ideals in $\mathcal{SH}^c_{G,(p)}$, using the geometric fixed point functors to pull back prime ideals from $\mathcal{SH}^c_{(p)}$ and applying results of the previous subsections to deduce that every prime ideal is obtained in this way. For any $(H,n) \in \Sub(G) \times \NN_{>0}$, we write
\[ 
P_G(H,n)=(\Phi^H)^{-1}(P(n))=\{ X\in \mathcal{SH}^c_{G,(p)}\mid  K(n-1)_*(\Phi^H(X))=0 \}. 
\]
Since $\Phi^H$ is symmetric monoidal, it follows that each $P_G(H,n)$ is again a prime.

We now show that every prime ideal in $\mathcal{SH}^c_{G,(p)}$ is of the form $P_G(H,n)$ for some closed subgroup $H$ and $n\in \NN_{>0}$, where $n$ is unique and $H$ is unique up to conjugacy. We first prove uniqueness, for which we make use of finite $G$-spectra of the following form: 

\begin{definition} Let $H\in \Sub(G)$ and $n\in \NN$. A finite $G$-spectrum $X$ is said to be \emph{of type} $(\subjconj{H},n)$ if $\type{X}(K)=n$ for all subgroups $K$ which are conjugate in $G$ to a subgroup of $H$ and $\type{X}(K)=\infty$ for all other $K$.
\end{definition}

Such finite $G$-spectra always exist, by the following simple construction: Let $Y$ be a non-equivariant finite type $n$ spectrum. We consider the finite $G$-spectrum $X=G/H_+\wedge i_G^*(Y)$, i.e., the smash product of the based $G$-space $G/H_+$ with the inflation of $Y$ to a $G$-spectrum.

\begin{lemma} For every non-equivariant finite type $n$ spectrum $Y$ the finite $G$-spectrum $X=G/H_+\wedge i_G^*(Y)$ is of type $(\subjconj{H},n)$.
\end{lemma}
\begin{proof} For a subgroup $K$ of $G$ we have an equivalence
\[ \Phi^K(X)\simeq (G/H)^K_+\wedge \Phi^K(i_G^*(Y))\simeq (G/H)^K_+\wedge Y,\]
which follows from the behaviour of geometric fixed points with respect to smash product, suspension spectra, and inflation, as recalled in \Cref{sec:eqhomotopy}.
Now if $K$ is conjugate to a subgroup $K^g$ of $H$, then $(G/H)^K$ is non-empty, since it contains the coset $g^{-1}H$. Hence $(G/H)^K_+$ contains $S^0$ as a retract, and it follows that the type of $(G/H)^K_+\wedge Y$ is equal to the type of $Y$, namely~$n$. If on the other hand $K$ is not conjugate to a subgroup of $H$, then the fixed points $(G/H)^K$ are empty, and hence $\Phi^K(X)$ is trivial.
\end{proof}

The existence of such finite $G$-spectra directly implies:

\begin{lemma}\label{lem:primerelations}
Let $H$ and $K$ be closed subgroups of $G$, and $n,m\in \NN_{>0}$. We have:
\begin{enumerate}
	\item If there is an inclusion $P_G(K,n)\subseteq P_G(H,m)$, then $n\geq m$ and $K$ is conjugate to a subgroup of $H$.
	\item If $P_G(K,n)=P_G(H,m)$, then $n=m$ and $K$ is conjugate to $H$.
\end{enumerate}
\end{lemma}
\begin{proof} Part (1): If $n<m$, then any finite $G$-spectrum of type $(\subjconj{G},n)$ lies in $P_G(K,n)$ but not in $P_G(H,m)$. If $K$ is not conjugate to a subgroup of $H$, then any finite $G$-spectrum of type $(\subjconj{H},m-1)$ lies in $P_G(K,n)$ but not in $P_G(H,m)$.

Part (2) is a direct consequence of Part (1), since a closed subgroup of a compact Lie group cannot be conjugate to a proper subgroup of itself. Indeed, suppose $H$ is conjugate to $H'\subseteq H$. Then by compactness, $H'$ has to be both closed and open in $H$, hence a union of finitely many components of $H$. But we also have a bijection $\pi_0(H)\cong \pi_0(H')$, hence $H'=H$ as desired.
\end{proof}

\begin{theorem} \label{thm:underlyingset}
Let $G$ be a compact Lie group. Then the assignment $(H,n)\mapsto P_G(H,n)$ defines a bijection 
\[ 
(\Sub(G)/G)\times \NN_{>0} \xrightarrow{\cong} \Spc(\mathcal{SH}^c_{G,(p)}). 
\]
\end{theorem}
\begin{proof} 
\Cref{lem:primerelations} shows that the map is injective, hence it suffices to see that every prime ideal $\wp$ is of the form $P_G(H,n)$ for some subgroup $H$ and $n\in \N$. By \Cref{cor:eqsupport}, we have
\[ 
\wp=\{ Y\in \mathcal{SH}^c_{G,(p)}\mid  \type{Y}\geq \type{\wp} \}=\bigcap_{H} P_G(H,\type{\wp}(H)),
\]
where we take $P_G(H,\type{\wp}(H))$ to mean the full category $\mathcal{SH}^c_{G,(p)}$ if $\type{\wp}(H)=0$. Since $\wp$ is prime, it is in particular not the full tt-ideal of all finite $G$-spectra, so there exists a subgroup $H$ such that $\type{\wp}(H)>0$. Without loss of generality, we can assume that $H$ is minimal with this property, since every descending chain of compact Lie groups stabilizes after finitely many steps. We claim that $\wp$ is equal to $P_G(H,\type{\wp}(H))$. If not, there would exist an element $X\in P_G(H,\type{\wp}(H))\setminus\wp$. Furthermore, any finite $G$-spectrum $Y$ of type $(\subjconj{H},0)$ does not lie in $\wp$, but we claim it is contained in $\bigcap_{H'\not\sim H} P_G(H',\type{\wp}(H'))$. Indeed, if $H'$ is conjugate to a proper subgroup of $H$, then $\type{\wp}(H')=0$ by the minimality of $H$ and hence $P_G(H',\type{\wp}(H'))$ is the ideal of all finite $G$-spectra. If $H'$ is not conjugate to a subgroup of $H$, then $\type{Y}(H')=\infty$ and hence $Y$ is contained in $P_G(H',n)$ for all $n$. However, the smash product $X\wedge Y$ lies in the intersection 
\[ 
P_G(H,\type{\wp}(H))\cap \bigcap_{H'\not \sim H} P_G(H',\type{\wp}(H'))=\wp,
\]
contradicting the fact that $\wp$ is prime. Hence, such an $X$ cannot exist, which implies that $\wp=P_G(H,\type{\wp}(H))$ as desired. 
\end{proof}

\subsection{The equivariant nilpotence theorem} \label{sec:nilpotence}

The goal of this subsection is to strengthen the observation made in the proof of \Cref{prop:eqsupport} that $\{\kappa_{(H,n)}\}_{(H,n) \in \Sub(G) \times \NN}$ detects ring objects to an equivariant version of the nilpotence theorem. This generalizes \cite[Thm.~4.15]{BS17b} from finite groups to arbitrary compact Lie groups; however, unlike their proof, our argument does not rely on the classification of prime ideals in $\SH_{G, (p)}$, but rather exploits properties of the geometric fixed point functors to reduce to the non-equivariant nilpotence theorem by Devinatz, Hopkins, and Smith.

\begin{lemma} Let $R$ be a homotopy $G$-ring spectrum (i.e., an object of $\Alg(\SH_{G})$)  and $x\in \pi_n^G(R)$ an element. Then $x$ is nilpotent if and only if the mapping telescope
\[ 
\tel(x)=\hocolim_{n\in \N} (R \xra{x\cdot} S^{-n}\wedge R \xra{x\cdot} S^{-2n}\wedge R \xra{x\cdot} \hdots) 
\]
is contractible.
\end{lemma}
\begin{proof} First we note that the homotopy ring $\pi_*^G(\tel(x))$ is isomorphic to the colimit
\[ 
\colim_{n\in \N}(\pi_*^G(R)\xra{x\cdot} \pi_*^G(R) \xra{x\cdot} \hdots) 
\]
Now, if $\tel(x)$ is contractible, then this colimit must be trivial. In particular, the element represented by the class of the element $1\in \pi_{0}^G(R)$ under the inclusion of the first term in the colimit system must be trivial. By standard facts about sequential colimits, this means that $x^n\cdot 1=x^n$ must be trivial for some $n\in \N$, and hence $x$ is nilpotent.

If $x$ is nilpotent, say $x^n=0$, then clearly the colimit above is trivial, since any composition of $n$ maps in the colimit system is trivial. So we find that $\pi_*^G(\tel(x))=0$. Furthermore, all restrictions $\res_H^G(x)$ to closed subgroups $H$ are also nilpotent, since restriction maps are ring homomorphisms. So we conclude that $\pi_*^H(\tel(\res_H^G(x)))$ is trivial for all closed subgroups $H$. Since we have an equivalence 
\[ 
\tel(\res_H^G(x)) \simeq \res_H^G(\tel(x)), 
\]
this shows that all the equivariant homotopy groups of $\tel(x)$ vanish and hence $\tel(x)$ is contractible. 
\end{proof}
\begin{remark} In fact, the statement holds more generally for $x\in \pi_V^G(R)$ an element in the $RO(G)$-graded homotopy ring of $R$, with the same proof.
\end{remark}

Recall that there is a comparison map
\[ 
\Phi^H\colon \pi_*^H(X)\to \pi_*(\Phi^H(X))=\Phi^H_*(X), 
\]
which is a ring map if $X$ is a homotopy $G$-ring spectrum. Given an element $x\in \pi_*^G(X)$, we also use the notation $\Phi^H(x)$ for the element $\Phi^H(\res_H^G(x))\in \Phi^H_*(X)$. 

\begin{prop} \label{prop:nilpotence} Let $R$ be a homotopy $G$-ring spectrum and $x\in \pi_*^G(X)$. Then $x$ is nilpotent if and only if all its geometric fixed points $\Phi^H(x)$ are nilpotent in $\Phi_*^H(R)$.
\end{prop}
\begin{proof}By the previous lemma we know that $x$ is nilpotent if and only if its mapping telescope $\tel(x)$ is contractible. But this is the case if and only if all the geometric fixed points $\Phi^H(\tel(x))$ are contractible. Since we have an equivalence $\Phi^H(\tel(x))\simeq \tel(\Phi^H(x))$, this in turn is equivalent to all the geometric fixed points $\Phi^H(x)$ being nilpotent, again by the previous lemma. This finishes the proof.
\end{proof}

Recall that $F(-,-)$ denotes the function $G$-spectrum between two $G$-spectra, see \cite[Sec. II.3]{LMSM86} or \cite[Sec. II.3]{MM02}. The function $G$-spectra are uniquely determined by being right adjoint to the smash product and make $\mathcal{SH}_{G,(p)}$ a closed symmetric monoidal category. In particular, as is generally the case for closed symmetric monoidal categories (see, for example, \cite[Secs. 1.5-1.8]{Kel82}), the function $G$-spectra $F(-,-)$ define a canonical enrichment of $\mathcal{SH}_{G,(p)}$ over itself, and the smash product canonically extends to an enriched bi-functor. As a formal consequence, for any pair of $G$-spectra $X$ and $Y$, one can form a homotopy $G$-ring spectrum $R_{X,Y}=\bigvee_n F(X^{\wedge n},Y^{\wedge n})$ with multiplication given by the enriched smash product
\[ 
F(X^{\wedge n},Y^{\wedge n})\wedge F(X^{\wedge m},Y^{\wedge m})\to F(X^{\wedge (n+m)},Y^{\wedge (n+m)}).
\]
Its homotopy ring is given by
\[ 
\pi_*^G(R_{X,Y})\cong \bigoplus_{n} [X^{\wedge n},Y^{\wedge n}]_*^G. 
\]
Now if $X$ is a finite $G$-spectrum, the geometric fixed points $\Phi^{H}(F(X^{\wedge n},Y^{\wedge n}))$ are naturally equivalent to the non-equivariant function spectrum $F(\Phi^H(X)^{\wedge n},\Phi^H(Y)^{\wedge n})$. One way to see this is the following: Since $X$ is dualizable, $F(X^{\wedge n},Y^{\wedge n})$ is canonically equivalent to $F(X^{\wedge n},\mathbb{S}_G)\wedge Y^{\wedge n}$ (see \cite[Thm. 4.17]{GM95a}). Now $F(X^{\wedge n},\mathbb{S}_G)$ is the dual of $X^{\wedge n}$, whose $H$-geometric fixed points are equivalent to the dual of $\Phi^{H}(X^{\wedge n})\simeq \Phi^{H}(X)^{\wedge n}$, since $\Phi^H$ is symmetric monoidal. Hence we obtain the above equivalence via
\begin{align*} \Phi^H(F(X^{\wedge n},Y^{\wedge n}))& \simeq \Phi^H(F(X^{\wedge n},\mathbb{S}_G))\wedge \Phi^H(Y^{\wedge n}) \\
&\simeq F(\Phi^H(X)^{\wedge n},\mathbb{S})\wedge \Phi^H(Y)^{\wedge n} \\
& \simeq F(\Phi^H(X)^{\wedge n},\Phi^H(Y)^{\wedge n}),\end{align*}
where the last step uses dualizability of $\Phi^H(X)^{\wedge n}$.

The above equivalences are compatible with taking smash products, so it follows that, for finite $X$, $\Phi^H(R_{X,Y})$ is equivalent as a homotopy ring spectrum to $R_{\Phi^H(X),\Phi^H(Y)}$. Furthermore, the geometric fixed point map takes a map $g\colon X^{\wedge n}\to Y^{\wedge n}$ to $\Phi^H(g)\colon\Phi^H(X)^{\wedge n}\to \Phi^H(Y)^{\wedge n}$. Applying \Cref{prop:nilpotence} to $R_{X,Y}$ hence implies:

\begin{cor} \label{cor:smashnilpotent}
A map $f\colon X\to Y$ of $G$-spectra, with $X$ finite, is smash-nilpotent if and only if $\Phi^H(f)\colon\Phi^H(X)\to \Phi^H(Y)$ is smash-nilpotent for all closed subgroups $H$.
\end{cor}

These results allow us to transport non-equivariant nilpotence theorems to the equivariant world, by testing on all geometric fixed points. In particular, we obtain the following equivariant version of the nilpotence theorem as a consequence of the non-equivariant nilpotence theorem:

\begin{theorem}
A map $f\colon X\to Y$ of $p$-local $G$-spectra, with $X$ finite, is smash-nilpotent if and only if $K(n)_*\Phi^H(f)\colon K(n)_*\Phi^H(X)\to K(n)_*\Phi^H(Y)$ is smash-nilpotent in the category of graded $K(n)_*$-modules for all closed subgroups $H$ of $G$ and all $n\in \NN$.
\end{theorem}

As mentioned previously, for finite groups $G$ this theorem was proven using tt-geometric techniques and induction on the order of $G$ by Balmer and Sanders.  We close this section with a remark about alternative versions of the equivariant nilpotence theorem. 

\begin{remark}
Let $MU_G$ be the homotopical $G$-equivariant complex cobordism spectrum in the sense of tom Dieck~\cite{tomdieck_eqmu}. For any closed subgroup $H \subseteq G$ the geometric fixed points $\Phi^HMU_G$ split into a non-trivial wedge of shifted copies of non-equivariant $MU$, see \cite[Thm.~4.10]{sinha_eqmu}, so it follows from \Cref{cor:smashnilpotent} and the classical nilpotence theorem~\cite{DHS88} for $MU$ that $MU_G$ detects smash-nilpotence of maps between finite $G$-spectra.

When $G=C_2$ one can also consider the real cobordism spectrum $MU_{\R}$, where $C_2$ acts on the underlying spectrum $MU$ via complex conjugation \cite{Lan68}. This one does not detect smash-nilpotence since its geometric fixed points $\Phi^{C_2}MU_{\R}\simeq MO$ only see $2$-primary information. For a concrete example of a non smash-nilpotent map between finite $C_2$-spectra which is nilpotent in $MU_{\R}$, let $C$ be the cofiber of the transfer map $S^0 \to (C_2)_+ \to S^0$ and consider the multiplication by $2$ map
\[
v=2\colon C \to C.
\]
We claim that $(MU_{\R})_*(v)$ is nilpotent, while $v$ is not. To this end, first observe that $\Phi^e(v)$ induces multiplication by $2$ on the mod-$2$ Moore spectrum $\Phi^eC \simeq S^0/2$ and hence is smash-nilpotent. Secondly, applying $C_2$-geometric fixed points yields the map
\[
\Phi^{C_2}(v) = 2\colon S^0 \vee S^1 \to S^0 \vee S^1,
\] 
which induces the null map after smashing with $H\Z/2$ and hence also $MO$. Therefore, the map $MO_*(\Phi^{C^2}(v))$ is nilpotent. However, multiplication by $2$ is not nilpotent on $\Phi^{C_2}C\simeq S^0 \vee S^1$ and consequently $v$ cannot be smash-nilpotent.

\end{remark}

\section{Type functions of finite $G$-spectra are locally constant} \label{sec:subtopology}

In this section we discuss an essential property of the type functions $\type{X}$ for finite $G$-spectra $X$: They are locally constant with respect to the Hausdorff topology on the space of conjugacy classes of closed subgroups $\Sub(G)/G$. This property plays a central role in the description of the Balmer topology, as we explain in \Cref{sec:topology}.

We now recall the topology on the set $\Sub(G)/G$. First we choose a bi-invariant metric $d$ on the Lie group $G$, and also write $d$ for the induced Hausdorff metric on $\Sub(G)$, defined via
\[ 
d(H,K)=\max(\sup(d(h,K)\mid  h\in H), \sup(d(H,k)\mid  k\in K)),
\]
where $d(h,K)$ denotes the distance from the point $h$ to the compact subset $K$, and likewise for $d(H,k)$. One can show that the resulting topology on $\Sub(G)$ is independent of the chosen metric on~$G$. The topology on $\Sub(G)/G$ is then defined as the quotient topology from the projection $\Sub(G)\to \Sub(G)/G$. We collect the following properties of the spaces $\Sub(G)$ and $\Sub(G)/G$:

\begin{prop} \label{prop:mz}
For any compact Lie group $G$, the space of subgroups $\Sub(G)$ is compact. If a sequence of subgroups $(H_i)_{i\in \N}$ converges to a subgroup $H$ of $G$ in $\Sub(G)$, then $H_i$ is conjugate to a subgroup of $H$ for almost all $i\in \N$. Moreover, the conjugation action of $G$ on $\Sub(G)$ is continuous, and the quotient $\Sub(G)/G$ is a compact totally-disconnected Hausdorff space.
\end{prop}
A proof is given in Propositions 5.6.1 and 5.6.2 of \cite{tD79}. The part about convergent sequences of subgroups is a consequence of the Montgomery--Zippin Theorem \cite{MZ42}.

\begin{example} The subgroups of the circle group $\T$ are given by the cyclic subgroups $C_n$ for $n\in \N$ and the full group $\T$. The Hausdorff topology on $\Sub(\T)$ is that of a converging sequence, i.e., the subspace $\{C_n\}_{n\in \N}$ is discrete and the sets $\{C_n\}_{n\geq m}\cup \{\T\}$ for $m\in \N$ form a neighbourhood basis for the point $\T$.
\end{example}

\begin{figure}[h]
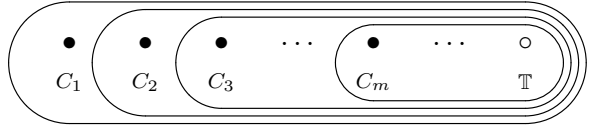

\[
\qquad\vcenter{\xy
%
%
%
(0,0)*{\bullet};
(10,0)*{\bullet};
(20,0)*{\bullet};
(30,0)*{\hdots};
(40,0)*{\bullet};
(50,0)*{\hdots};
(60,0)*{\circ};
%
%
(0,-5)*{\scriptstyle C_1};
(10,-5)*{\scriptstyle C_2};
(20,-5)*{\scriptstyle C_3};
(40,-5)*{\scriptstyle C_m};
(60,-5)*{\scriptstyle \T};
%
%
(40,-2.5)*\cir(5,5){l^r};
{\ar@{-} (40,2.5)*{};(60,2.5)*{}};
(60,-2.5)*\cir(5,5){r_l};
{\ar@{-} (60,-7.5)*{};(40,-7.5)*{}};
%
%
(20,-2.5)*\cir(6,6){l^r};
{\ar@{-} (20,3.5)*{};(60,3.5)*{}};
(60,-2.5)*\cir(6,6){r_l};
{\ar@{-} (60,-8.5)*{};(20,-8.5)*{}};
%
%
(10,-2.5)*\cir(7,7){l^r};
{\ar@{-} (10,4.5)*{};(60,4.5)*{}};
(60,-2.5)*\cir(7,7){r_l};
{\ar@{-} (60,-9.5)*{};(10,-9.5)*{}};
%
%
(0,-2.5)*\cir(8,8){l^r};
{\ar@{-} (0,5.5)*{};(60,5.5)*{}};
(60,-2.5)*\cir(8,8){r_l};
{\ar@{-} (60,-10.5)*{};(0,-10.5)*{}};
\endxy}
\]
\caption{A schematic illustration of $\Sub(\T)$.}
\end{figure}

As claimed above, we have:
\begin{prop} \label{prop:isotropylocconstant} For every finite $G$-spectrum $X$, the type function
\[ 
\type{X}\colon\Sub(G)/G \to \NN 
\]
is locally constant.
\end{prop}
\begin{proof} 
Let $X$ be a finite $G$-spectrum. By \Cref{lem:finiteg}, $X$ is of the form $S^{-V}\wedge \Sigma^{\infty} A$, where $A$ is a homotopy retract of the $p$-localization of a finite based $G$-CW complex $B$. Since smashing with a representation sphere does not change the type function, we can reduce to showing the proposition in the case where $V=0$ and hence $X\simeq \Sigma^{\infty}A$. We have to show that each $\type{X}^{-1}(n)$ is open. Let $H\in \Sub(G)$ such that $\type{X}(H)=n$. If the class of $H$ did not lie in the interior of $\type{X}^{-1}(n)$, there would be a sequence $H_i$ converging to $H$ with each $H_i\notin \type{X}^{-1}(n)$. By the Montgomery--Zippin theorem, we can assume that the $H_i$ are subgroups of $H$. Now the restriction $\res_H^G B$ to an $H$-space admits the structure of a finite $H$-cell complex, and thus the isotropy of its points consist of only finitely many conjugacy classes of subgroups of $H$. Let $K_1,\hdots,K_m,H$ be a choice of representatives for these conjugacy classes, where the full group $H$ is among the isotropy groups since $\res_H^GB$ has a fixed basepoint. It follows that if $K\subseteq H$ is not $H$-subconjugate to one of the $K_j$, then the restriction map $B^H\to B^K$ is a homeomorphism. Since $A$ is a homotopy retract of the $p$-localization of $B$, this means that for these $K$ the restriction map $A^H\to A^K$ is a homotopy equivalence. In particular, $\type{X}(K)=\type{X}(H)$ for all such $K$, since 
\[ \Phi^K(X)\simeq \Sigma^{\infty}(A^K)\simeq \Sigma^{\infty}(A^H) \simeq \Phi^H(X).\]
Thus, each $H_i$ above must be subconjugate to one of the $K_j$ and consequently some $K_j$, say $K_1$, is superconjugate to infinitely many of the $H_i$. But this implies that $d(H,H_i)\geq d(H,K_1)$ for all these $i$, contradicting the facts that (1) $K_1$ is a proper subgroup of $H$ and (2) the $H_i$ converge to~$H$.
\end{proof}

\begin{remark}
When $G$ is finite, the Hausdorff topology on $\Sub(G)$ is discrete and thus every function on it is automatically locally constant. In contrast to this, it will turn out in \Cref{thm:criterion} below that the topology on $\Sub(G)$ is one of the novel ingredients in the study of the Balmer spectrum for general compact Lie groups. 
\end{remark}

It is not generally true that the type functions of arbitrary tt-ideals $I$ are locally constant, as we will see in \Cref{thm:criterion} and \Cref{ex:notfinitegen}. However, we will need that they can be approximated by type functions $\type{X}$ for $X\in I$, in the following sense:
\begin{prop} \label{prop:approx} Let $G$ be a compact Lie group, $I$ a tt-ideal of $\mathcal{SH}_{G,(p)}^c$, and $f\colon\Sub(G)/G\to \NN$ a locally constant function such that $f\geq \type{I}$. Then there exists an $X\in I$ such that $f\geq \type{X}$.
\end{prop}
\begin{proof} For every $H\in \Sub(G)/G$ we choose an object $Y_H\in I$ such that $\type{Y_H}(H)\leq f(H)$, which exists by assumption. We note that since both $f$ and $\type{Y_H}$ are locally constant, there exists an open neighbourhood $U_H$ of $H$ on which both $f$ and $\type{Y_H}$ are constant, and hence $\type{Y_H}\leq f$ on that neighbourhood. The open sets $U_H$ cover $\Sub(G)/G$, which is compact. Hence there exists a finite subcover $U_{H_1},\hdots,U_{H_n}$. We now define $X$ via
\[ 
X=\bigvee_{i=1}^n Y_{H_i}. 
\]
Then on each $U_{H_i}$, we have $\type{X}\leq \type{Y_{H_i}}\leq f$. Since the $U_{H_i}$ cover $\Sub(G)/G$, it follows that $\type{X}\leq f$ on all of $\Sub(G)/G$, which shows that $X$ has the desired property. 
\end{proof}

\section{The Balmer topology and classification of tt-ideals} \label{sec:topology}
For finite groups $G$, it is shown in \cite{BS17b} that the Balmer topology on $\Spc(\mathcal{SH}^c_{G,(p)})$ and the classification of tt-ideals can be expressed in terms of the poset structure on the Balmer spectrum. In this section we show that in the case of compact Lie groups, the Balmer topology and the classification of tt-ideals are again determined by the poset structure on the set of prime ideals, provided one also remembers the Hausdorff topology on $\Sub(G)/G$. We will return to the study of the poset structure in \Cref{sec:primeposet}.

We first recall from \Cref{cor:eqsupport} that every tt-ideal $I$ is determined by its type function $\type{I}$. Hence, we obtain a classification of tt-ideals if we can describe which functions $\Sub(G)/G\to \NN$ occur as the type function of a tt-ideal. It turns out that the type functions are detected by the following two properties, the first relying on the poset structure on the Balmer spectrum and the second on the topology on $\Sub(G)/G$:

\begin{definition}
We call a function $f\colon\Sub(G)/G\to \NN$ \emph{admissible} if $P_G(K,f(K))$ is not contained in $P_G(H,f(H)+1)$ for all subgroups $K,H$ with $f(H)<\infty$.
\end{definition}

\begin{definition}
A function $f\colon\Sub(G)/G\to \NN$ is \emph{upper semi-continuous} if $f^{-1}([0,n))$ is open for all $n\in \NN$.
\end{definition}
We have:
\begin{theorem}\label{thm:criterion} Let $G$ be a compact Lie group and  $f\colon\Sub(G)/G\to \NN$ be a function. Then the following hold:
\begin{enumerate}
	\item The function $f$ is of the form $\type{I}$ for some tt-ideal $I$ if and only if it is admissible and upper semi-continuous.
	\item The function $f$ is of the form $\type{X}$ for some finite $G$-spectrum $X$ if and only if it is admissible and locally constant.
\end{enumerate}
\end{theorem}

\begin{remark} 
We will see later in \Cref{lem:admissible} that when $G$ is abelian every admissible function is automatically upper semi-continuous, but this is not the case in general, for example for $G=O(2)$, see \Cref{rem:o2}.
\end{remark}

\begin{example} \label{ex:notfinitegen}
An example of a tt-ideal $I$ that is not finitely-generated is the prime ideal $\wp=P_\T(\T,1)$ for the circle group $\T$. For any finite subgroup $C_n$ of $\T$, the finite $\T$-spectrum $\Sigma^{\infty}(\T/C_n)_+$ has trivial $\T$-geometric fixed points, hence it is contained in $\wp$, but its $C_n$-geometric fixed points have type $0$. It follows that the type function $\type{\wp}$ sends all finite subgroups $C_n$ to $0$ and the full group $\T$ to $1$, hence it is not locally constant.
\end{example}

Before we prove the theorem, we need two preparatory lemmas:

\begin{lemma} \label{lem:lift} Let $X$ and $Y$ be finite $G$-spectra and $f\colon\Phi^G(X)\to \Phi^G(Y)$ a map between their geometric fixed points. Then there exists a $G$-representation $V$ satisfying $V^G=0$ and a $G$-map $\widetilde{f}\colon X\to Y\wedge S^{V}$ such that
\[ 
\Phi^G(\widetilde{f})\colon \Phi^G(X)\to \Phi^G(Y\wedge S^V)\simeq \Phi^G(Y) 
\]
equals $f$.
\end{lemma}
\begin{proof} 
There is a natural isomorphism (for which finiteness of $X$ is not required)
\[ [\Phi^G(X),\Phi^G(Y)]\cong [X,\widetilde{E}P\wedge Y]^G, \]
where $\widetilde{E}P$ is the cofiber of the map $EP_+\to S^0$ and $EP$ is a universal space for the family of proper subgroups of $G$. This isomorphism is a consequence of the fact that $\widetilde{E}P\wedge -$ is a smashing localization on the homotopy category of $G$-spectra, whose category of local objects is equivalent to the homotopy category of non-equivariant spectra via taking geometric fixed points, see \cite[Cor. II.9.6]{LMSM86}, \cite[Thm. XVI.6.6]{may_equivariant96} or \cite[Thm. 6.11]{MNN17}. The isomorphism assigns to a $G$-map $X\to \widetilde{E}P\wedge Y$ its induced map on geometric fixed points, using that $\Phi^G(\widetilde{E}P)\simeq S^0$. A model for $\widetilde{E}P$ is given by the homotopy colimit over $S^V$, where $V$ runs through the poset of all finite dimensional subrepresentations of a complete $G$-universe $\mathcal{U}_G$ which have trivial $G$-fixed points; this is a consequence of the fact that for every proper closed subgroup $H$ of $G$ there exists a $G$-representation $V$ such that $V^G=0$ but $V^H\neq 0$, see \cite[Prop. III.4.2]{BtD85}. Hence,
\[ [\Phi^G(X),\Phi^G(Y)]\cong [X,\hocolim_{V\subset \mathcal{U}_G,V^G=0} (S^V\wedge Y)]^G. \]
Now if $f\colon\Phi^G(X)\to \Phi^G(Y)$ is given as above, it corresponds to a $G$-map 
\[ F\colon X\to \hocolim_{V\subset \mathcal{U}_G,V^G=0} (S^V\wedge Y).\]
Since $X$ is finite, the map $F$ factors through a finite stage $S^V\wedge Y$, yielding the desired~$\widetilde{f}$.
\end{proof}

\begin{lemma} \label{lem:shift} Let $K$ and $H$ be subgroups of $G$, and $n\in \NN$ and $m\in \N$. If $P_G(K,n)$ is not contained in $P_G(H,m+1)$, then there exists a finite $G$-spectrum $X$ such that $\type{X}(K)\geq n$ and $\type{X}(H)=m$.
\end{lemma}
\begin{proof} Since $P_G(K,n)$ is not contained in $P_G(H,m+1)$, there exists a finite $G$-spectrum $Y$ such that $\type{Y}(K)\geq n$ and $\type{Y}(H)\leq m$. If $\type{Y}(H)=m$, we are done. If not, we will explain how to use $Y$ to construct another finite $G$-spectrum $Y'$ such that $\type{Y'}(K)\geq n$ and $\type{Y'}(H)=\type{Y}(H)+1$. Iterating this process leads to the desired finite $G$-spectrum $X$.

The construction of $Y'$ goes as follows. We first consider the case where $H$ equals the ambient group $G$. By the periodicity theorem of Hopkins and Smith \cite[Thm.~9]{HS98}, there exists a $v_{\type{Y}(G)}$-self map $f\colon\Sigma^{k} \Phi^G(Y)\to \Phi^G(Y)$, i.e., a self-map which induces an isomorphism on $K(\type{Y}(G))_*$ and is trivial on all other Morava $K$-theories. Applying \Cref{lem:lift} to this $f$, we obtain a $G$-map $\widetilde{f}\colon\Sigma^{k} Y\to Y\wedge S^V$ for some $G$-representation $V$ with trivial fixed points, such that the geometric fixed points~$\Phi^H(\widetilde{f})$ give back $f$. Then the cofiber $Y'=\cone(\widetilde{f})$ has the desired properties: The geometric fixed points $\Phi^G(Y')$ are given by the cofiber of $f$, whose type equals $\type{Y}(G)+1$, since $f$ is a $v_{\type{Y}(G)}$-self map. Furthermore, since $Y'$ lies in the tt-ideal generated by $Y$, the type of $\Phi^K(Y')$ cannot be smaller than that of $\Phi^K(Y)$.

If $H$ is a proper subgroup of $G$, one repeats the same argument for $G=H$ applied to the restriction $\res_H^G(Y)$. This leads to a finite $H$-spectrum $\widetilde{Y}$, defined as the cofiber of an $H$-map $\widetilde{f}\colon\res_H^G(Y)\to \res_H^G(Y)\wedge S^V$, where $V$ is an $H$-representation with trivial fixed points. By the same argument as in the absolute case, the type of its $H$-geometric fixed points is given by $\type{Y}(H)+1$, and the type of $\Phi^J(\widetilde{Y})$ is at least that of $\Phi^J(Y)$ for all other subgroups $J$ of $H$. Let $Y'=\Ind_H^G(\widetilde{Y})$. Then \Cref{lem:induced} shows that $\type{Y'}(H)=\type{Y}(H)+1$, and that 
\begin{align*} \type{Y'}(K)& \geq \min(\type{\widetilde{Y}}(K')\ |\ K'\subseteq H, K'\sim K) \\ & \geq \min(\type{Y}(K') |\ K'\subseteq H, K'\sim K) \\
& =\type{Y}(K), \end{align*}
where the last equality follows from the fact that $\type{Y}$ takes the same value on all $K'$ that are conjugate to $K$, because $Y$ is a $G$-spectrum. Hence $Y'$ has the desired property, which finishes the proof.
\end{proof}

\begin{proof}[Proof of \Cref{thm:criterion}] We start by showing that if $I$ is a tt-ideal in $\mathcal{SH}_{G,(p)}^c$, then the type function $\type{I}$ is admissible and upper semi-continuous. Let $H$ and $K$ be two closed subgroups of $G$ such that $\type{I}(H)<\infty$. Further, let $X\in I$ be a finite $G$-spectrum with $\type{X}(H)=\type{I}(H)$. Then $X\notin P_G(H,\type{I}(H)+1)$, but $X\in I\subseteq P_G(K,\type{I}(K))$. Thus, $P_G(K,\type{I}(K))$ cannot be contained in $P_G(H,\type{I}(H)+1)$ and hence $\type{I}$ is admissible. Furthermore, by \Cref{prop:isotropylocconstant} we know that $\type{X}$ is locally constant for every finite $G$-spectrum $X$, thus $\type{X}^{-1}([0,n))$ is in particular open, for all $n\in \NN$. Hence, 
\[ 
\type{I}^{-1}([0,n))=\bigcup_{X\in I}(\type{X}^{-1}([0,n)) 
\]
is also open, and thus $\type{I}$ is upper semi-continuous.

Now let $f\colon \Sub(G)/G\to \NN$ be an admissible and upper semi-continuous function. We wish to construct a tt-ideal $I$ such that $\type{I}=f$. For this we let $H$ be a closed subgroup of $G$. Our first aim is to construct an $X_H\in \mathcal{SH}^c_{G,(p)}$ such that $\type{X_H}(H)=f(H)$ and $\type{X_H}(K)\geq f(K)$ for all other subgroups $K$ of $G$. We first fix such a $K$ in addition to $H$. Since $P_G(K,f(K))$ is not contained in $P_G(H,f(H)+1)$, \Cref{lem:shift} shows that there exists an $X_{H,K}\in \mathcal{SH}^c_{G,(p)}$ such that $\type{X_{H,K}}(H)=f(H)$ and $\type{X_{H,K}}(K)\geq f(K)$. Since $f$ is upper semi-continuous and $\type{X_{H,K}}$ is locally constant (\Cref{prop:isotropylocconstant}), there is an open neighbourhood $U_{H,K}$ of $K$ on which $\type{X_{H,K}}$ is constant with value $\type{X_{H,K}}(K)=f(K)$ and $f$ is bounded above by $f(K)$. The space $\Sub(G)/G$ is compact, thus it is covered by finitely many $U_{H,K_1},\hdots,U_{H,K_n}$. We then define $X_H$ as the smash product
\[ 
X_H=X_{H,K_1}\wedge \hdots \wedge X_{H,K_n}. 
\]
Then $\type{X_H}(K)=\max\{\type{X_{H,K_i}}(K)\mid  i=1,\hdots,n\}$. Since $\type{X_{H,K_i}}(H)=f(H)$ for all $i$, it follows that also $\type{X_H}(H)=f(H)$. Moreover, for $K\in U_{H,K_i}$, we have $\type{X_{H,K_i}}(K)= f(K_i)\geq f(K)$. Thus, $\type{X_H}(K)\geq f(K)$. The $U_{H,K_i}$ cover $\Sub(G)/G$, so it follows that $\type{X_H}\geq f$, as desired. Now we set 
\[ 
I=\tti{\{ X_H\mid  H\in \Sub(G)/G\}}. 
\]
Then $\type{I}\geq f$, since $\type{X_H}\geq f$ for all $H$. On the other hand we have $\type{X_H}=f(H)$, hence $\type{I}\leq f$. Therefore $\type{I}$ is equal to $f$, which concludes the proof of the first part of the theorem. 

For the second part of the theorem we need to show that a function $f\colon \Sub(G)/G\to \NN$ is equal to $\type{X}$ for some finite $G$-spectrum $X$ if and only if it is admissible and locally constant. We first show the `only if' direction. We know that $\type{X}$ is locally constant for every finite $G$-spectrum $X$, again by \Cref{prop:isotropylocconstant}. From the first part of the theorem it follows that $\type{X}$ is also admissible, since $\type{X}=\type{\tti{X}}$. For the other direction let $f\colon \Sub(G)/G\to \NN$ be an admissible and locally constant function. Again using the first part it follows that there exists a tt-ideal $I$ such that $f=\type{I}$, since locally constant in particular implies upper semi-continuous. We can now apply \Cref{prop:approx} to $I$ and the locally constant function $\type{I}$, and see that there is an $X\in I$ such that $\type{X}\leq \type{I}$. But, $X\in I$ also implies $\type{X}\geq \type{I}$. Hence, $\type{X}=\type{I}=f$, which concludes the proof.
\end{proof}

We obtain the following characterization of the Balmer topology on $\Spc(\mathcal{SH}^c_{G,(p)})$: 
\begin{cor}\label{cor:topbasis} 
For $G$ a compact Lie group, the Balmer topology on $\Spc(\mathcal{SH}^c_{G,(p)})$ has as a basis the open sets
\[ 
\{ P_G(H,n)\mid  n\leq f(H) \},
\]
where $f$ ranges through all locally constant admissible functions $\Sub(G)/G\to \NN$.
\end{cor}
\begin{proof} By definition, the Balmer topology has as basis the open sets 
\[ 
U_X=\{ \wp \in \Spc(\mathcal{SH}^c_{G,(p)})\mid  X\in \wp \}, 
\]
where $X$ ranges through all finite $G$-spectra. By \Cref{thm:underlyingset}, all prime ideals are of the form $P_G(H,n)$. Moreover, by definition, $X$ lies in $P_G(H,n)$ if and only if $n\leq \type{X}(H)$. Thus we have
\[ 
U_X= \{ P_G(H,n)\mid  n\leq \type{X}(H) \}. 
\]
Hence, $U_X$ depends only on the function $\type{X}$. By the previous theorem, we know that $\type{X}$ ranges through all admissible locally constant functions. This finishes the proof.
\end{proof}

Since admissibility only depends on the poset structure of $\Spc(\mathcal{SH}^c_{G,(p)})$, this shows that the Balmer topology is indeed determined by the inclusions among prime ideals together with the topology on $\Sub(G)/G$, as claimed.

\section{Inclusions of prime ideals}\label{sec:primeposet}

We continue to let $G$ be a compact Lie group and work $p$-locally. In order to obtain a complete description of the topology on the Balmer spectrum $\Spc(\mathcal{SH}_{G,(p)}^c)$ and a classification of the tt-ideals, we are hence left to determine the poset structure on $\Spc(\mathcal{SH}_{G,(p)}^c)$. By \Cref{lem:primerelations}, we know that a prime $P_G(K,m)$ can only be contained in $P_G(H,n)$ if $K$ is conjugate to a subgroup of $H$ and $m\geq n$. Moreover, if $P_G(K,m)\subseteq P_G(H,n)$, then also $P_G(K,m+1)\subseteq P_G(H,n)$, since $P_G(K,m+1)\subseteq P_G(K,m)$. The remaining problem can thus be phrased as the following question:

\begin{question} \label{quest} Given an inclusion of subgroups $K\subseteq H$ and $n\in \NN_{>0}$, what is the minimal $i\in \NN$, if any, for which there is an inclusion 
\[ 
P_G(K,n+i)\subseteq P_G(H,n)? 
\]
Phrased more concretely, given a finite $G$-spectrum $X$ such that $\Phi^K(X)$ is of type $\geq n+i$, is $\Phi^H(X)$ necessarily of type $\geq n$?
\end{question}

This problem appears to be difficult and is open in general, already for finite groups $G$ where it is linked to the blue-shift phenomenon for generalized Tate constructions \cite[Sec.~9]{BS17b}. In this section we show that the question can be reduced to the case where $K\subseteq H$ is a `$p$-subcotoral' inclusion. We also give lower and upper bounds for the minimal $i$ in \Cref{quest}, building on the previously obtained bounds for finite groups from \cite{BS17b} and \cite{BHN^+17}.
In \Cref{sec:abelian} we then use these bounds to derive a complete classification for abelian compact Lie groups. For a general compact Lie group $G$, we obtain an answer away from finitely many critical primes $p$, those that divide the `order' $|G|$ of $G$, as we explain in \Cref{sec:order}.

We first need to introduce some terminology. An inclusion $K\subseteq H$ of subgroups of $G$ is said to be \emph{cotoral} if it is normal and the quotient is a torus. For every subgroup $K$ there exists a maximal cotoral inclusion $K\subseteq \omega_G(K)$ in $G$, unique up to conjugacy. Moreover, the group $\omega_G(K)$ has finite Weyl group and is hence cotorally maximal, i.e., there exists no proper cotoral inclusion of the form $\omega_G(K)\subset H$. In fact, the function $K\mapsto \omega_G(K)$ defines a continuous retraction of $\Sub(G)/G$ onto the closed subspace of conjugacy classes of subgroups with finite Weyl group (see \cite[Def.~1.2f.]{FO05} for a discussion of all these properties). Given a subgroup $H$ of $G$, we further write $\mathcal{O}(H)$ for the minimal normal subgroup of $H$ with quotient a finite $p$-group. We say that a subgroup $K\subseteq H$ is \emph{$p$-subcotoral} if $\mathcal{O}(K)$ sits inside $\mathcal{O}(H)$ and the inclusion $\mathcal{O}(K)\subseteq \mathcal{O}(H)$ is cotoral. When $G$ is finite, this notion reduces to the notion of $p$-subnormal subgroup studied by Balmer and Sanders, see \cite[Lem.~3.3]{BS17b}.

Then the following theorem summarizes what we know about inclusions of primes: 

\begin{theorem} \label{thm:inclusions} The following hold for subgroups $K$ and $H$ of $G$ and $n,m \in \NN_{>0}$:
\begin{enumerate}[i)]
	\item If there is an inclusion $P_G(K,m)\subseteq P_G(H,n)$, then $K$ is $G$-conjugate to a $p$-subcotoral subgroup of $H$ and $m\geq n$.
	\item If there is an inclusion $P_G(K,1)\subseteq P_G(H,1)$, then $K$ is $G$-conjugate to a cotoral subgroup of $H$.
	\item	If $K$ is $G$-conjugate to a cotoral subgroup of $H$, then $P_G(K,n)\subseteq P_G(H,n)$ for all $n$.
	\item If $K$ is $G$-conjugate to a normal subgroup of $H$ with quotient a cyclic $p$-group, then $P_G(K,n+1)\subseteq P_G(H,n)$ for all $n$.
	\item	If $K$ is $G$-conjugate to a $p$-subcotoral subgroup of $H$, then $P_G(K,\infty)\subseteq P_G(H,\infty)$.
	\item If $K\subseteq H$ is normal in $G$ and $H/K$ is an elementary abelian $p$-group of rank $k$, then \[ P_G(K,n+k-1) \not\subseteq P_G(H,n) \]
for all $k,n\in \N$.
\end{enumerate}
\end{theorem}

The proof is given in \Cref{sec:burnside} and \Cref{sec:cotoral}. Using the change of groups properties of \Cref{prop:changeofgroups} below, Items $iv)$ and $vi)$ are immediate consequences of the main result of \cite{BHN^+17}, with the case $H/K\cong \Z/p$ previously shown in \cite{BS17b}. 

\begin{remark}
Note that the conclusion of $vi)$ translates into the existence of a finite $G$-spectrum~$X$ with the property that $\type{X}(K) \ge n+k-1$ and $\type{X}(H)< n$. Under the condition that $H/K$ is elementary abelian, suitable complexes can be formally constructed from the Arone--Lesh complexes~\cite{AL17}, which thus form the crucial ingredient in our proof of  $vi)$. We currently do not know how to extend this result beyond the case of abelian groups. 
\end{remark}

We note that $ii)$ and $iii)$ imply that $P_G(K,1)\subseteq P_G(H,1)$ if and only if $K$ is $G$-conjugate to a cotoral subgroup of $H$. This gives back the computation of the Balmer spectrum of the homotopy category of finite rational $G$-spectra of \cite{Gre17}. Parts $iii)$ and $iv)$ can be used to obtain many more inclusions of primes by inducting on subnormal chains $K=H_1\subseteq H_2\subseteq \hdots \subseteq H_n=H$ with each quotient $H_{i+1}/H_i$ a torus or a cyclic $p$-group.

\subsection{Change of groups} We begin by proving the following change of groups properties regarding \Cref{quest}.
\begin{prop} \label{prop:changeofgroups}
Let $G$ be a compact Lie group and consider $\Spc(\mathcal{SH}_{G,(p)}^c)$.
\begin{enumerate}[i)]
	\item There is an inclusion $P_G(K,n)\subseteq P_G(H,m)$ if and only if there exists an element $g\in G$ such that the conjugate $K^g$ is a subgroup of $H$ and $P_H(K^g,n)\subseteq P_H(H,m)$.
	\item	If $K'$ is a normal subgroup of $G$ contained in $K$ and $H$, then there is an inclusion $P_G(K,n)\subseteq P_G(H,m)$ if and only if there is an inclusion $P_{G/K'}(K/K',n)\subseteq P_{G/K'}(H/K',m)$.
	\item Let $K'\subseteq K\subseteq H$ be subgroups of $G$ such that $K$ is normal in $G$. If $P_G(K',n)\subseteq P_G(H,m)$, then also $P_G(K,n)\subseteq P_G(H,m)$.
\end{enumerate}
\end{prop}
The first part can be interpreted as a `Going-Up' theorem for restriction to subgroup functors, see \cite[Prop. 6.9]{BS17b} in the finite group case.

\begin{proof} $i)$: If $P_H(K^g,n)\subseteq P_H(H,m)$, then also $P_G(K^g,n)\subseteq P_G(H,m)$, since the question of whether a finite $G$-spectrum $X$ lies in $P_G(H,m)$ only depends on the restriction $\res_H^G(X)$ to an $H$-spectrum. Since $K^g$ is conjugate to $K$, it follows that also $P_G(K,n)\subseteq P_G(H,m)$. We now show that if $P_H(K^g,n)\not\subseteq P_H(H,m)$ for all conjugates $K^g$ of $K$ which are subgroups of $H$, then also $P_G(K,n)\not\subseteq P_G(H,m)$. We want to apply induction $\Ind_H^G(-)$ to a suitable $H$-spectrum $X$. Recall from \Cref{lem:induced} that the type of $\Phi^K(\Ind_H^G X)$ is at least as large as the minimum of the types of all $\Phi^{K^g}(X)$ for $G$-conjugates $K^g$ contained in $H$, and that the type of $\Phi^H(\Ind_H^G X)$ is the same as that of $\Phi^H(X)$. It is hence sufficient to construct a finite $H$-spectrum $X$ such that $\Phi^H(X)$ has type~$<m$ and $\Phi^{K^g}(X)$ has type $\geq n$ for all conjugates $K^g$. For every conjugate $K^g$, we choose a finite $H$-spectrum $X_{K^g}$ with the property that $\Phi^{K^g}(X_{K^g})$ has type $\geq n$ and $\Phi^H(X_{K^g})$ has type~$<m$, which exists by assumption. Note that, up to $H$-conjugacy, there are only finitely many $K^g\subseteq H$ (see \Cref{rem:subconjugacy}). Hence, we can consider the smash product
\[ 
X=\bigwedge_{(K^g\subseteq H)} X_{K^g}. 
\]
This has the desired property, since the type of a smash product is the maximum of the types of the smash factors.

$ii)$: This follows from the fact that $\Spc(\mathcal{SH}^c_{G/K',(p)})$ sits inside $\Spc(\mathcal{SH}^c_{G,(p)})$ as a retract, as the subspace of all $P_G(H,-)$ with $H$ containing $K'$. The inclusion $\Spc(\mathcal{SH}^c_{G/K',(p)})\to \Spc(\mathcal{SH}^c_{G,(p)})$ is induced by the refined version of the geometric fixed point functor $\Phi^{K'}$ which remembers the $G/K'$-action. The projection $\Spc(\mathcal{SH}^c_{G,(p)})\to \Spc(\mathcal{SH}^c_{G/K'})$ is induced by the inflation functor $i_{K'}\colon\mathcal{SH}^c_{G/K'}\to \mathcal{SH}^c_{G,(p)}$ (\Cref{sec:inflation}).

$iii)$: The third statement also follows from the previous argument, this time applied to the projection $p\colon G\to G/K$. The self-map $(\Phi^K)^*\circ (i_K)^*$ of $\mathcal{SH}^c_{G,(p)}$ sends the ideal $P_G(H,m)$ to itself, since $H$ contains the normal subgroup $K$. On the other hand, $(i_K)^*$ sends $P_G(K',n)$ to $P_{G/K}(p(K'),n)$, as a consequence of the equivalence \eqref{eq:inflation} of \Cref{sec:inflation}. Since $K'$ is contained in $K$, the image $p(K')$ is the trivial subgroup $K/K$. Thus 
\[
(\Phi^K)^*((i_K)^*(P_G(K',n)))=P_G(K,n). 
\]
As the composite $(\Phi^K)^*\circ (i_K)^*$ preserves inclusions, this shows the claim.
\end{proof}

If one is willing to study \Cref{quest} simultaneously for all compact Lie groups, \Cref{prop:changeofgroups} allows us to reduce to the case where $H$ equals the ambient group $G$ and $K$ contains no non-trivial subgroup normal in $H$. In \Cref{sec:reduce} we explain how this can be used to reduce to the case where $H$ is an extension
\[ 1\to A\to H\to P\to 1, \]
of a finite $p$-group $P$ by an abelian compact Lie group $A$ with $|\pi_0(A)|$ coprime to $p$. In the case where $K$ is a normal subgroup of $H$, it can be further reduced to $H$ a finite $p$-group.

\subsection{The spectrum of the Burnside ring} \label{sec:burnside}
We now start with the proof of \Cref{thm:inclusions}. Inspired by \cite{BS17b}, we use the comparison map to the spectrum of the Burnside ring $A(G)=\pi_0^G(\mathbb{S}_G)$ to rule out certain inclusions between the primes. The spectrum of the Burnside ring was computed by Dress in the case of finite groups \cite{Dre69} and by tom Dieck \cite{tD75} for compact Lie groups. We now recall this computation, or more precisely its $p$-local version. For every closed subgroup $H$, the geometric fixed point map induces a ring homomorphism on endomorphism rings
\[ 
\varphi^H\colon A(G)_{(p)} \cong \pi_0^G\mathbb{S}_G \to \pi_0\Phi^H(\mathbb{S}_G)\cong \pi_0(\mathbb{S})\cong \mathbb{Z}_{(p)}. 
\]
These homomorphisms are called the \emph{marks} of $A(G)_{(p)}$.

\begin{remark}
 The marks were originally studied by tom Dieck in \cite{tD75}. There he defined the Burnside ring of a compact Lie group as the ring of certain equivalence classes of closed differentiable $G$-manifolds. The mark homomorphism $\varphi^H$ assigns to such a $G$-manifold $X$ the Euler characteristic of its fixed space $X^H$ (in fact the equivalence relation is exactly defined as considering two $G$-manifolds to be equivalent if these Euler characteristics are the same for all closed subgroups $H$ of $G$). In \cite{tD75c} tom Dieck proved that his version of the Burnside ring agrees with $\pi_0^G\mathbb{S}_G$. The isomorphism sends a closed $G$-manifold $X$ to the composite
\[ \mathbb{S}_G\to F(\Sigma^{\infty}X_+,\Sigma^{\infty} X_+)\simeq \Sigma^{\infty}X_+\wedge F(\Sigma^{\infty} X_+,\mathbb{S}_G) \to \mathbb{S}_G, \]
using that $\Sigma^{\infty} X_+$ is dualizable in $\mathcal{SH}_{G,(p)}$. This construction is called the categorical Euler characteristic or trace, see \cite[Secs. III.7 and V.1]{LMSM86}. Since geometric fixed points $\Phi^H$ are symmetric monoidal, it follows that they send the categorical Euler characteristic of $\Sigma^{\infty}X_+$ to the categorical Euler characteristic of $\Phi^H(\Sigma^{\infty} X_+)\simeq \Sigma^{\infty} X^H_+$ in the non-equivariant stable homotopy category. It is a classical fact that this categorical Euler characteristic equals the usual Euler characteristic for any finite CW-complex. Hence tom Dieck's definition of the mark homomorphisms agrees with the one given above (see also \cite[Sec. V.2]{LMSM86} for more details on this discussion). 
\end{remark}

It turns out that the marks determine the spectrum of $A(G)_{(p)}$. For the two primes $(0)$ and $(p)$ of $\mathbb{Z}_{(p)}$, we denote by $q(H,1)$ and $q(H,\infty)$ the respective pullback along $\varphi^H$. Recall that, given a subgroup $H$ of $G$, we write $\omega_G(H)$ for a maximal cotoral extension of $H$ in $G$ and $\mathcal{O}(H)$ for the minimal normal subgroup of $H$ whose quotient is a finite $p$-group. Finally we call $\omega_G(\mathcal{O}(H))$ the \emph{$p$-perfection} of $H$ in $G$ (see \cite{FO05}). Note that, unlike $\mathcal{O}(H)$, the $p$-perfection depends on the ambient group $G$.
\begin{theorem}[({\cite[Sec. 4]{tD75},\cite[Sec. 12]{Fau08}})] \label{thm:burnside} Let $G$ be a compact Lie group. All prime ideals of $A(G)_{(p)}$ are of the form $q(H,1)$ or $q(H,\infty)$. They are related as follows:
 \begin{enumerate}
	\item $q(H,1)=q(H',1)$ if and only if $\omega_G(H)$ is conjugate to $\omega_G(H')$, and there are no proper inclusions of the form $q(H,1)\subset q(H',1)$.
	\item $q(H,\infty)=q(H',\infty)$ if and only if $\omega_G(\mathcal{O}(H))$ is conjugate to $\omega_G(\mathcal{O}(H'))$, and there are no proper inclusions of the form $q(H,\infty)\subset q(H',\infty)$.
	\item $q(H,1)\subseteq q(H',\infty)$ if and only if $q(H,\infty)=q(H',\infty)$.
	\item There are no inclusions of the form $q(H,\infty)\subseteq q(H',1)$.
\end{enumerate}
\end{theorem} 
\begin{proof}
The fact that all prime ideals are of this form is \cite[Thm. 4]{tD75}. By \cite[Prop. 8]{tD75}, the primes of the form $q(H,1)$ only depend on the conjugacy class of $\omega_G(H)$ and the primes of the form $q(H,\infty)$ only depend on the conjugacy class of $\omega_G(\mathcal{O}(H))$. Combining this with \cite[Props. 9 and 12]{tD75} gives Part $(i)$. Part $(ii)$ can be derived from \cite[Prop. 16]{tD75}, but is stated more explicitly in this form in \cite[Lem. 12.5]{Fau08}. Part $(iii)$ is \cite[Prop. 9]{tD75}, while Part $(iv)$ follows from the fact that the element $p\in A(G)_{(p)}$ lies in $q(H',1)$ but not in $q(H,\infty)$, for any pair of closed subgroups $H,H'$ of $G$.
\end{proof}
The comparison map $\psi\colon \Spc(\mathcal{SH}^c_{G,(p)})\to \Spec(A(G)_{(p)})$ is defined to send $\wp\in \Spc(\mathcal{SH}^c_{G,(p)})$ to the prime ideal of all $\alpha\in A(G)_{(p)}$ such that the cofiber $C(\alpha)$ is not contained in $\wp$, see \cite{Bal10}. This computes as follows:

\begin{prop} Given $P_G(H,n) \in \Spc(\mathcal{SH}^c_{G,(p)})$,we have
\[ 
\psi(P_G(H,n))=\begin{cases} q(H,1) & \text{ if } n=1 \\
															q(H,\infty) & \text{ if } n>1. \end{cases}
\]
\end{prop}
\begin{proof} The proof builds on the commutative comparison diagram
\[ 
\xymatrix{ \Spc(\mathcal{SH}^c_{G,(p)}) \ar[r]^-{\psi} & \Spec(A(G)_{(p)}) \\
							\Spc(\mathcal{SH}^c_{(p)}) \ar[r]^-{\psi} \ar[u]^{(\Phi^H)^*}& \Spec(\Z_{(p)}). \ar[u]_{(\varphi^H)^*}  }
\]
Since $P_G(H,n)$ is by definition equal to $(\Phi^H)^*(P(n))$, it is enough to compute $\psi(P(n))$. First note that if $x\in \Z_{(p)}$ is non-trivial, then it becomes a unit in $\Q$, and hence $K(0)_*(C(x))=H\Q_*(C(x))=0$, and thus $x\notin \psi(P(1))$. It follows that $\psi(P(1))$ is the $0$-ideal, as claimed. Now, if $x\in \Z_{(p)}$ is divisible by $p$, then its image in $K(n-1)_*$ for $n>1$ is zero, and hence $C(x)\notin P(n)$, thus $x\in \psi(P(n))$. This shows that for all $n>1$ we have $(p)\subseteq \psi(P(n))$, and hence $(p)= \psi(P(n))$, since $(p)$ is maximal. The statement now follows from the commutativity of the above diagram and the facts that $q(H,1)=(\varphi^H)^*((0))$ and $q(H,\infty)=(\varphi^H)^*((p))$.\end{proof}

Since the comparison map is inclusion reversing, we see:
\begin{cor} If $P_G(K,n)\subseteq P_G(H,m)$ is an inclusion in $\Spc(\mathcal{SH}^c_{G,(p)})$, then $\omega_G(\mathcal{O}(K))$ is conjugate to $\omega_G(\mathcal{O}(H))$. If $n=1$, then $\omega_G(K)$ is conjugate to $\omega_G(H)$.
\end{cor}

In order to prove Parts $i)$ and $ii)$ of \Cref{thm:inclusions}, one combines this corollary with the first part of \Cref{prop:changeofgroups} on change of groups: If $P_G(K,n)\subseteq P_G(H,m)$, then there exists a $G$-conjugate $K^g\subseteq H$ such that $P_H(K^g,n)\subseteq P_H(H,m)$. Hence, by the corollary, $\omega_H(\mathcal{O}(K^g))\sim\omega_H(\mathcal{O}(H))=\mathcal{O}(H)$ and it follws that $K^g$ is $p$-subcotoral in $H$, as desired. If $n=1$, then even $\omega_H(K^g)\sim \omega_H(H)=H$ and $K^g$ is cotoral in $H$. In summary, this proves Parts $i)$ and $ii)$ of \Cref{thm:inclusions}:

\begin{cor}\label{cor:psubcotoral} If $P_G(K,n)\subseteq P_G(H,m)$, then $K$ is $G$-conjugate to a $p$-subcotoral subgroup of $H$ and $n\geq m$. If $n=1$, then $K$ is even $G$-conjugate to a cotoral subgroup of $H$.
\end{cor}

\subsection{Inclusions for $p$-cotoral and $p$-subcotoral subgroups} \label{sec:cotoral} 
While the previous subsection was about establishing necessary conditions for the inclusion of prime ideals, we now turn to sufficient ones. We begin with Parts $iii)$ and $iv)$ of \Cref{thm:inclusions}, which via \Cref{prop:changeofgroups} reduce to the following:

\begin{prop} \label{prop:cotoral} Let $K\subseteq H$ be a normal inclusion of closed subgroups in $G$:
\begin{enumerate}
	\item If $H/K$ is a torus, then $P_G(K,n)\subseteq P_G(H,n)$ for all $n \in  \NN$.
	\item If $H/K$ is a cyclic $p$-group, then $P_G(K,n+1)\subseteq P_G(H,n)$ for all $n \in \NN$.
\end{enumerate}
\end{prop}
As mentioned above, the second item is \cite[Thm. 2.6]{BHN^+17}. In order to prove the first, we will use a small computation of the geometric $\T$-fixed points of Morava $K$-theories proved in \Cref{ssec:phibkn} below.

\begin{proof}[Proof of \Cref{prop:cotoral}] It only remains to prove the first item. Since the enhanced version of the geometric fixed points $\Phi^K(X)$ is a finite $H/K$-spectrum satisfying $\Phi^{H/K}(\Phi^K(X))\simeq \Phi^H(X)$, it suffices to consider the case where $K$ is the trivial group and $H=G=\T^k$ is a torus. By induction on $k$ we can further reduce to $k=1$.

Hence we are given a finite $\T$-spectrum $X$ such that $K(n)_*(\Phi^{\{e\}}(X))=0$, and we want to show that $K(n)_*(\Phi^{\T}(X))$ is also trivial. We consider the function spectrum $F_\T(X,\underline{K(n)})$ from $X$ into the $\T$-Borel spectrum for $K(n)$ (with trivial naive $\T$-action), see \Cref{sec:borelcompletion}. Then $F_\T(X,\underline{K(n)})$ is itself Borel complete, since there is an equivalence 
\[ F_\T(E\T_+,F_\T(X,\underline{K(n)}))\simeq F_\T(X,F_\T(E\T_+,\underline{K(n)}))\simeq F_\T(X,\underline{K(n)}). \]
Moreover, the underlying non-equivariant spectrum of $F_\T(X,\underline{K(n)})$ is $F(\Phi^{\{e\}}(X),K(n))$ (since $\Phi^{\{e\}}(X)\simeq \res_e^\T(X)$), which is contractible by assumption. Hence $F_\T(X,\underline{K(n)})$ is equivariantly contractible and therefore its geometric fixed points $\Phi^\T(F_\T(X,\underline{K(n)}))$ are also contractible. Since $X$ is finite, we can compute these geometric fixed points as
\[ 
0 \simeq \Phi^\T(F_\T(X,\underline{K(n)}))\simeq F_\T(\Phi^\T(X),\mathbb{S})\wedge \Phi^\T(\underline{K(n)}), 
\]
see the discussion preceding \Cref{cor:smashnilpotent}. By \Cref{thm:moravageom}, $\Phi^\T(\underline{K(n)})$ is a non-trivial $K(n)$-module and hence contains $K(n)$ as a retract. So we conclude that $F_\T(\Phi^\T(X),\mathbb{S})\wedge K(n)$ is trivial and hence so is $\Phi^\T(X)\wedge K(n)$, using the K{\"u}nneth formula for Morava $K$-theory.

To obtain the statement also for $n=\infty$, one uses that $P_G(H,\infty)$ is the intersection of all $P_G(H,n)$ with $n<\infty$.
\end{proof}

We are left with showing Part $v)$ of \Cref{thm:inclusions}:
\begin{prop}
Let $K$ and $H$ be two closed subgroups of $G$. If $K$ is $p$-subcotoral in $H$, then there is an inclusion $P_G(K,\infty)\subseteq P_G(H,\infty)$.
\end{prop}
\begin{proof}
The key to the proof is the following statement, due to tom Dieck \cite[Thm. 4]{tD75}: If $K\subseteq H$ is $p$-subcotoral and $K\neq H$, then either $K$ has infinite Weyl group $W_HK$ or the order of $W_HK$ is divisible by $p$. This leads to the following algorithm to construct $H$ out of $K$: First replace $K$ by $\omega_H(K)$, which we can do since we know that $P_G(K,\infty)\subseteq P_G(\omega_H(K),\infty)$ and $\omega_H(K)$ is again $p$-subcotoral in $H$. Hence, it is enough to consider the case where $K$ has finite Weyl group. If $K=H$, we are done. If not, we can find a normal extension $K\subseteq K_1$ (inside $H$) of index $p^k$ with $k\geq 1$. Then $P_G(K,\infty)\subseteq P_G(K_1,\infty)$ and $\mathcal{O}(K_1)=\mathcal{O}(K)$. Hence, $K_1$ is also $p$-subcotoral in $H$ and we can replace $K$ by $K_1$. This way we obtain a subnormal chain $K\subseteq K_1\subseteq K_2\subseteq \hdots$. If at some point we have $K_i=H$, we are done. If not, we consider the closure $L$ of the union of the $K_i$. Since we have $P_G(K,\infty)\subseteq P_G(K_i,\infty)$ for all $i$, it follows that we also have $P_G(K,\infty)\subseteq P_G(L,\infty)$, since the type function $\type{X}$ of every finite $G$-spectrum is locally constant and the $K_i$ converge to $L$ in $\Sub(G)$.

Further, we claim that $L$ is also $p$-subcotoral in $H$. There are different ways to see this. One is to use tom Dieck's theorem~(\Cref{thm:burnside}, Part (ii)) that $L\subseteq H$ is $p$-subcotoral if and only if the primes $q(L,\infty)$ and $q(H,\infty)$ are the same in the Burnside ring of $H$. We know that $q(H,\infty)$ is equal to $q(K_i,\infty)$ for all $i$, and a similar argument as in the proof of \Cref{prop:isotropylocconstant} shows that then also $q(H,\infty)$ is equal to $q(L,\infty)$, using that for every $H$-representation $V$ we have $V^L=V^{K_i}$ for almost all $i$. Hence, we can replace $K$ by $L$ and proceed. Since the dimension of $L$ is larger than that of $K$, this process must eventually terminate at the full group $H$. This finishes the proof.
\end{proof}

\subsection{Geometric fixed points of Borel-equivariant Morava $K$-theories}\label{ssec:phibkn}
We now give the computation needed for the proof of the first part of \Cref{prop:cotoral}, and also explain that this argument is special for tori. Let $\underline{K(n)}$ denote the $G$-equivariant Borel spectrum for the $n$-th Morava $K$-theory $K(n)$, see \Cref{sec:borelcompletion}. Then we have:
\begin{theorem}\label{thm:moravageom}
Suppose $G$ is a compact Lie group, $n\ge 0$, and $r \ge 0$, then
\[
\pi_*\Phi^{G}(\underline{K(n)}) \cong
	\begin{cases}
		K(n)_*\llbracket x_1,\ldots,x_r\rrbracket(x_1^{-1},\ldots,x_r^{-1}) & \text{if } G \cong \T^{\times r}, \\
		0 & \text{otherwise.}
	\end{cases}
\]
\end{theorem}
\begin{proof} 
We start with the case where $G=\T$ is the circle group. Since $K(n)$ is a complex oriented cohomology theory and using \cite[Prop.~3.20]{GM95a}, we can compute $\Phi_*^{\T}(\underline{K(n)})$ as the localization of 
\[ 
\pi_*^\T(\underline{K(n)})\cong K(n)_*(B\T)\cong K(n)_*\llbracket x\rrbracket 
\]
at Euler classes $e_V$, where $V$ ranges through a countable collection $\mathcal{C}$ of complex $\T$-representations satisfying two properties:
\begin{enumerate}
	\item The fixed points $V^{\T}$ are trivial for all $V\in \mathcal{C}$.
	\item	For every proper closed subgroup $H$ of $\T$ there exists a representation $V\in \mathcal{C}$ such that $V^H\neq 0$.
\end{enumerate}
This fact uses that under these assumptions the unit sphere in the representation $\sum_{\N} \sum_{V\in \mathcal{C}} V$ is a model for the universal space for the family of proper subgroups. For $n>0$, let $\rho_n$ be the one-dimensional complex representation given by the $n$-fold tensor power of the tautological representation $\rho_1$. Then the collection $\{\rho_n\}_{n>0}$ satisfies the two properties above. The Euler class of $\rho_n$ is given by the $n$-series $[n]_F(x)$, where $F$ is the formal group law associated to the complex orientation of $K(n)$. We claim that it is in fact enough to invert the single element $x$, i.e., the $1$-series $[1]_F(x)$, and hence that we have an isomorphism
\[ 
\Phi_*^{\T}(\underline{K(n)})\cong K(n)_*\llbracket x\rrbracket(x^{-1}) 
\]
as claimed. To show that, it is enough to see that the leading term of $[n]_F(x)$ is a unit for all $n>1$. We first consider the case where $n=p^k$ is a power of $p$. If $k=1$, then $[p]_F(x)$ has leading term $v_n\cdot x^{p^n}$, and $v_n$ is a unit in $K(n)_*$. For larger $k$, we can use the formula $[p^k]_F(x)=[p]_F([p^{k-1}]_F(x))$ to see that the leading coefficient of $[p^k]_F(x)$ is given by $v_n$ multiplied with the $p^n$-fold power of the leading coefficient of the $p^{k-1}$-series, which is a unit by induction. Finally we can write a general $n$ as $p^k\cdot q$ with $q$ not divisible by~$p$. Then $[n]_F(x)= [q]_F([p^k]_F(x))$. Since $[q]_F(x)$ has leading term $q\cdot x$, we see that the leading coefficient of $[n]_F(x)$ is given by $q$ times the leading coefficient of $[p^k]_F(x)$ and hence a unit.

Next we deal with the case $G=\T^{\times r}$ of a higher rank torus. There is an isomorphism
\[ 
\pi_*^{\T^{\times r}}(\underline{K(n)})\cong K(n)_*(B\T^{\times r})\cong K(n)_*\llbracket x_1,\hdots,x_r\rrbracket,
\]
and the ring $\pi_*\Phi^{\T^{\times r}}(\underline{K(n)})$ is obtained from this by inverting all Euler classes of non-trivial irreducible complex representations of $\T^{\times r}$. The $x_i$ are the Euler classes associated to the projections $\T^{\times r}\to \T$, hence we need to see that after inverting all $x_i$ the Euler class of any non-trivial irreducible complex representation becomes invertible. Every such irreducible representation $V$ is of the form $V_1\otimes_{\C} \hdots \otimes_{\C} V_r$, where $V_i$ is the pullback of an irreducible representation of $\T$ along the $i$-th projection. Since $V$ is non-trivial, one of the $V_i$ needs to be non-trivial, which by symmetry we can assume to be $V_1$. Since the Euler class of a tensor product of $1$-dimensional representations is given by the formal sum of the Euler classes of the factors, we see that
\[ e_V=e_{V_1}+_F \hdots +_F e_{V_r}. \]
Now each $e_{V_i}$ is a power series in the variable $x_i$. Furthermore, since $V_1$ is non-trivial we have seen in the rank $1$ case above that $e_{V_1}$ becomes invertible after adjoining $x_1^{-1}$. But this implies that $e_V$ also becomes invertible in $K(n)_*\llbracket x_1,\ldots,x_r\rrbracket(x_1^{-1})$, and therefore in particular in $K(n)_*\llbracket x_1,\ldots,x_r\rrbracket(x_1^{-1},\ldots,x_r^{-1})$, since when viewed as a power series over $K(n)_*\llbracket x_1\rrbracket$ it is given by $e_{V_1}$ plus higher order terms. This finishes the proof of the torus case.

We now turn to the second part and show that if $G$ is not a torus, the $G$-geometric fixed points of $\underline{K(n)}$ are trivial. We begin with the case where $G$ is a non-trivial finite group, which follows directly from the Tate vanishing of $K(n)$: By \cite[Thm. 1.1]{GS96}, the $G$-spectrum $\underline{K(n)}\wedge \tilde{E}G$ is trivial, where $\tilde{E}G$ is the cofiber of $EG_+ \to S^0$. But $\Phi^G(\tilde{E}G)$ is equivalent to the sphere $S^0$, so 
\[ \Phi^G(\underline{K(n)})\simeq \Phi^G(\underline{K(n)}\wedge \tilde{E}G)\simeq 0. \]
Next we consider the case where $G$ is disconnected but not necessarily discrete. The quotient map $p\colon G\to \pi_0(G)$ induces a ring map
\[ 
K(n)^*(B\pi_0(G))\to K(n)^*(BG) 
\]
sending an Euler class $e_V$ to the pulled back Euler class $e_{p^*(V)}$. As $p$ is surjective, the $G$-fixed points of $p^*(V)$ are trivial if the $\pi_0(G)$-fixed points of $V$ are trivial. It follows that the geometric fixed point map $K(n)^*(BG)\to \Phi^G_*(\underline{K(n)})$ factors through the localization of $K(n)^*(BG)$ at the Euler classes of the form $e_{p^*(V)}$ with $V^{\pi_0(G)}=0$. But this localization is trivial, since it is an algebra over 
\[ 
K(n)^*(B\pi_0(G))[e_V^{-1}\ |\ V^{\pi_0(G)}=0]\cong \Phi^{\pi_0(G)}_*(\underline{K(n)}),
\]
which we have just seen to be trivial.

The remaining case of $G$ connected holds in fact more generally  than for Morava $K$-theory: 
\begin{lemma} Let $E$ be any Borel-complete $G$-spectrum for $G$ a connected compact Lie group which is not a torus. Then $\Phi^G(E)$ is contractible.
\end{lemma}
\begin{proof} Write $T$ for a maximal torus in $G$ and $N = N_G(T)$ for the normalizer of $T$ in $G$. Since $N$ is a proper subgroup of $G$, the Peter--Weyl theorem implies that there exists an orthogonal representation $V$ of $G$ with the property that $V^G=0$ but $V^N\neq 0$. We claim that the map
\[ E\wedge e_V\colon E\simeq E\wedge S^0\to E\wedge S^V \] 
is trivial. Since $e_V$ induces the identity on geometric fixed points and consequently so does $E\wedge e_V$, this implies that $\Phi^G(E)$ is contractible.

The proof of the claim is a standard transfer argument, using that the Euler characteristic of the quotient $G/N$ is $1$. We consider the composition
\[ G_+ \to (G\times G/N)_+\to G_+ \]
obtained by smashing the transfer map $G/G_+\to G/N_+$ and the restriction map $G/N_+\to G/G_+$ with $G_+$. This composition is given by multiplication with the Euler characteristic of $G/N$ and hence an equivalence. It follows that $X\to X\wedge G/N_+\to X$ is an equivalence for any spectrum $X$ build out of free cells, and in particular for $X$ the suspension spectrum of $EG_+$. Using that $E\simeq F(EG_+,E)$ by assumption, it follows that $E$ is naturally a retract of the function spectrum $F(G/N_+,E)$. Hence it suffices to show that $F(G/N_+,E)\wedge e_V\simeq F(G/N_+,E\wedge e_V)$ is trivial. But this follows from the fact that $e_V$ is trivial when restricted to $N$-spectra (as $V^{N}\neq 0$), since the functor $F(G/N_+,-)$ is equivalent to the composition of the restriction to $N$-spectra followed by coinduction back up to $G$-spectra.
\end{proof}
This finishes the proof of \Cref{thm:moravageom}.
\end{proof}

\begin{example}
We note that the corresponding vanishing statement for the Tate construction does not hold. Indeed if $G$ is a compact Lie group which is not finite then $\pi_*(H\Q^{hG})=H^{-*}(BG; \Q)$ is non-zero in arbitrarily large negative degrees. The norm sequence \cite[Page 5]{GM95} shows that $H^{-*}(BG)=H\Q^{hG}_*\rightarrow H\Q^{tG}_*$ is an isomorphism in negative degrees, so in particular $H\Q^{tG}_*\neq 0$.\\
\end{example}

\subsection{Reduction to extensions of $p$-groups by abelian groups} \label{sec:reduce}

In the final part of this section we show that the change of group results of \Cref{prop:changeofgroups} can be used to reduce \Cref{quest} to extensions of finite $p$-groups $P$ by abelian compact Lie groups $A$ with $|\pi_0(A)|$ coprime to $p$.

\begin{prop}
Let $G$ be a compact Lie group. The poset structure on $\Spc(\SH_{G,(p)}^c)$ is determined by the poset structure on $\Spc(\SH_{Q,(p)}^c)$ for all subquotients $Q$ of $G$ which are extensions of a finite $p$-group by an abelian compact Lie group $A$ with $|\pi_0(A)|$ coprime to $p$.
\end{prop}
\begin{proof}
Recall that by \Cref{prop:changeofgroups} Part $i)$ and \Cref{cor:psubcotoral}, a complete understanding of the poset structure on $\Spc(\SH_{G,(p)}^c)$ will follow from understanding inclusions $P_H(K,n+i)\subseteq P_H(H,n)$, where $K$ is a $p$-subcotoral subgroup of $H$. We now fix such a $p$-subcotoral extension $K\subseteq H$ and consider the commutator subgroup $C=[\mathcal{O}(H),\mathcal{O}(H)]$ of the minimal normal subgroup $\mathcal{O}(H)$ of $H$ of $p$-power index. Since $K$ is $p$-subcotoral, the group $\mathcal{O}(K)$ sits cotorally inside $\mathcal{O}(H)$. Hence, it must contain the commutator subgroup $C$. It follows that so does $K$. Moreover, since $\mathcal{O}(H)$ is normal in $H$ and  $C$ is preserved by any automorphism of $\mathcal{O}(H)$, it follows that $C$ is also normal in $H$. Thus we can apply Part $ii)$ of \Cref{prop:changeofgroups} to find that there is an inclusion $P_H(K,n+i)\subseteq P_H(H,n)$ if and only if there is an inclusion $P_{H/C}(K/C,n+i)\subseteq P_{H/C}(H/C,n)$.

Now the group $H/C$ sits in the short exact sequence
\[ 
1\to \mathcal{O}(H)/C \to H/C\to H/\mathcal{O}(H) \to 1. 
\]
The group $P=H/\mathcal{O}(H)$ is a finite $p$-group. Moreover, $A=\mathcal{O}(H)/C$ is an abelian compact Lie group whose group of components $\pi_0(A)$ is $p$-perfect (as that of $\mathcal{O}(H)$ is). Thus, the order of $\pi_0(A)$ is coprime to $p$. This yields the desired reduction.
\end{proof}

\begin{remark} If $G$ is finite, the poset structure on $\Spc(\SH_{G,(p)}^c)$ is in fact determined by the poset structure on $\Spc(\SH_{Q,(p)}^c)$ for all subquotients $Q$ that are $p$-groups, see \cite[Prop. 6.11]{BS17b}. This is because for every $p$-subnormal inclusion $K\subseteq H$ the $p$-perfection $\mathcal{O}(H)$ is a normal subgroup of $H$ that is contained in $K$  and such that $H/\mathcal{O}(H)$ is a $p$-group. By \Cref{prop:changeofgroups}, there is an inclusion $P_H(K,n+i)\subseteq P_H(H,n)$ if and only if there is an inclusion $P_{H/\mathcal{O}(H)}(K/\mathcal{O}(H),n+i)\subseteq P_{H/\mathcal{O}(H)}(H/\mathcal{O}(H),n)$, reducing the problem to the poset structure on $\Spc(\SH_{H/\mathcal{O}(H),(p)}^c)$.

One could hope that for compact Lie groups the poset structure can be reduced to the poset structure for those subquotients which are extensions of finite $p$-groups by tori (i.e., $p$-toral groups), but we have been unable to see this. The problem is that given a $p$-subcotoral inclusion $K\subseteq H$, in general there exists no subgroup $K'\subseteq K$ that is normal in $H$ and for which $H/K'$ is a $p$-toral group.
\end{remark}

\section{Abelian compact Lie groups} \label{sec:abelian}
The results of the previous section allow us to give a complete description of the Balmer spectrum for abelian compact Lie groups $A$. Recall from \Cref{sec:topology} that the topology on $\Spc(\mathcal{SH}^c_{A,(p)})$ and the classification of its tt-ideals are determined by the set of admissible functions $f\colon\Sub(A)\to \NN$ and the Hausdorff topology on $\Sub(A)$. We say that a compact Lie group is \emph{p-toral} if it is an extension of a finite $p$-group by a torus, and an inclusion $K\subseteq H$ is \emph{$p$-cotoral} if it is normal with $p$-toral quotient $H/K$. Furthermore, given a finite abelian $p$-group $A'$, we write $rk_p(A')$ for the rank of $A'$, i.e., the dimension of $A'\otimes \Z/p$ over $\Z/p$. We then have:

\begin{theorem}  \label{thm:abelian}
Let $A$ be an abelian compact Lie group. A function $f\colon\Sub(A)/A \to \NN$ is admissible if and only if for all $p$-cotoral inclusions $K\subseteq H$ of closed subgroups of $A$ the inequality
\begin{equation} \label{eq:abelian} f(H)+rk_p(\pi_0(H/K)) \geq f(K) \tag{$\star$} \end{equation}
holds.
\end{theorem}
For later reference, we call Property ($\star$) the \emph{$p$-cotoral rank property}. 
\begin{proof} Recall that $f$ was defined to be admissible if for all subgroups $K$ and $H$ there is no inclusion $P_A(K,f(K))\subseteq P_A(H,f(H)+1)$ if $f(H)<\infty$.  By Part $i)$ of \Cref{thm:inclusions}, we know that such an inclusion can only happen if $K$ is a $p$-subcotoral subgroup of $H$. We claim that for abelian groups, $K$ being $p$-subcotoral in $H$ is equivalent to the quotient $H/K$ being $p$-toral: If there is a cotoral inclusion $\mathcal{O}(K) \subseteq \mathcal{O}(H)$, then $\mathcal{O}(K)\subseteq H$ is a $p$-toral extension, and hence so is $K\subseteq H$ since the quotient of a $p$-toral group is again $p$-toral. Conversely, if $H/K$ is $p$-toral, then so is $H/\mathcal{O}(K)$ as an extension of a $p$-toral group by a finite $p$-group. Thus, $H/\omega_H(\mathcal{O}(K))$ is a finite $p$-group and $\omega_H(\mathcal{O}(K))$ contains $\mathcal{O}(H)$ as a finite index subgroup. Furthermore, $K/(K\cap \mathcal{O}(H))$ is a finite $p$-group, since it is isomorphic to a subgroup of $H/\mathcal{O}(H)$. Thus, $K\cap \mathcal{O}(H)$ contains $\mathcal{O}(K)$. It follows that $\mathcal{O}(H)/\mathcal{O}(K)$ is a finite index subgroup of the torus $\omega_H(\mathcal{O}(K))/K$, hence $\mathcal{O}(H)=\omega_H(\mathcal{O}(K))$. This proves the claim.

Hence, it remains to show that for a $p$-cotoral inclusion $K\subseteq H$, we have 
\[ P_A(K,n+rk_p(\pi_0(H/K)))\subseteq P_A(H,n) \]
and 
\[ P_A(K,n+rk_p(\pi_0(H/K))-1)\not \subseteq P_A(H,n).\]
The first statement follows from the fact that there is a chain of inclusions
\[ 
K\subseteq \omega_H(K)=H_0 \subset H_1 \subset \hdots \subset H_{rk_p(\pi_0(H/K))}=H, 
\]
where each $H_i/H_{i-1}$ is $p$-power cyclic. Hence, by Parts $iii)$ and $iv)$ of \Cref{thm:inclusions}, we find that
\begin{align*}
P_A(K,n+rk_p(\pi_0(H/K))) &\subseteq P_A(\omega_H(K),n+rk_p(\pi_0(H/K))) \\
& \subseteq P_A(H_1,n+rk_p(\pi_0(H/K))-1) \\
& \subseteq \hdots \subseteq P_A(H,n),
\end{align*}
as desired. Finally, that $P_A(K,n+rk_p(\pi_0(H/K))-1)\not \subseteq P_A(H,n)$ follows from Part $vi)$ of \Cref{thm:inclusions} and Part $iii)$ of \Cref{prop:changeofgroups} applied to the subgroup chain $K \subseteq \omega_H(K) \subseteq H$.
\end{proof}

Hence, by \Cref{thm:criterion}, the tt-ideals of $\Spc(\mathcal{SH}^c_{A,(p)})$ correspond bijectively to upper semi-continuous functions $\Sub(A)\to \NN$ satisfying Property $(\star)$ above. In fact it turns out that Property $(\star)$ implies upper semi-continuity:
\begin{lemma} \label{lem:admissible} Let $A$ be an abelian compact Lie group and $f\colon\Sub(A)\to \NN$ a function such that $f(K)\leq f(H)$ for every cotoral inclusion $K\subseteq H$. Then $f$ is upper semi-continuous.
\end{lemma}
\begin{proof}
Let $n\in \NN$ and $H$ a closed subgroup of $A$ such that $f(H)< n$. We have to show that some open neighbourhood of $H$ is also contained in $f^{-1}([0,n))$. We claim that the set of subgroups $K$ of $H$ for which $K\subseteq H$ is a cotoral inclusion forms an open neighbourhood of $H$ in $\Sub(A)$. Note that for any such $K$ we have $f(K)\leq f(H)< n$ by assumption, hence $K\in f^{-1}([0,n))$. Thus we are done once we have shown the claim. Now, by the Montgomery--Zippin theorem (see \Cref{prop:mz}) the set of all subgroups of $H$ is open in $\Sub(A)$, hence it suffices to show that the set of cotoral subgroups $K\subseteq H$ is open in $\Sub(H)$. But this follows from the fact that the projection $H\to \pi_0(H)$ induces a continuous map $\Sub(H)\to \Sub(\pi_0(H))$ and the cotoral subgroups are exactly the ones that surject onto $\pi_0(H)$.
\end{proof}

\begin{cor}\label{cor:abelian} 
Let $A$ be an abelian compact Lie group. Then we have: 
\begin{enumerate}[i)]
\item The tt-ideals of $\Spc(\mathcal{SH}^c_{A,(p)})$ correspond bijectively to functions $f\colon\Sub(A)\to \NN$ satisfying the $p$-cotoral rank property  $(\star)$ above.
\item Given a function $f\colon\Sub(A)\to \NN$, there exists a finite $G$-spectrum $X$ with $\type{X}=f$ if and only if $f$ satisfies the $p$-cotoral rank property  $(\star)$ above and is locally constant.
\end{enumerate}
\end{cor}
\begin{proof} Every function $f:\Sub(A)\to \NN$ satisfying the $p$-cotoral rank property in particular satisfies the conditions for \Cref{lem:admissible}. Hence, the corollary is a direct consequence of \Cref{thm:criterion} and \Cref{thm:abelian} and this lemma.
\end{proof}

\begin{example}[$A=\mathbb{T}$] The subgroups of $\mathbb{T}$ are given by the finite cyclic subgroups $C_n$ for $n\geq 1$, and the full group $\mathbb{T}$. Therefore, we can identify $\Sub(\mathbb{T})$ with $\{1,2,\hdots\}\cup \{\infty\}$. The tt-ideals of $\Spc(\mathcal{SH}^c_{\mathbb{T},(p)})$ hence correspond to functions 
\[ 
f\colon\{1,2,\hdots\}\cup \{\infty\}\to \NN 
\]
such that $f(\infty) \geq f(n)$ for all $n$, and $f(p^k\cdot n)+1\geq f(n)$ for all $n$ and $k$. The tt-ideal corresponding to such a function is finitely generated if and only if $f(n)=f(\infty)$ for almost all~$n$. See \Cref{fig:spccircle} for an illustration.
\end{example}
\begin{figure}[h]
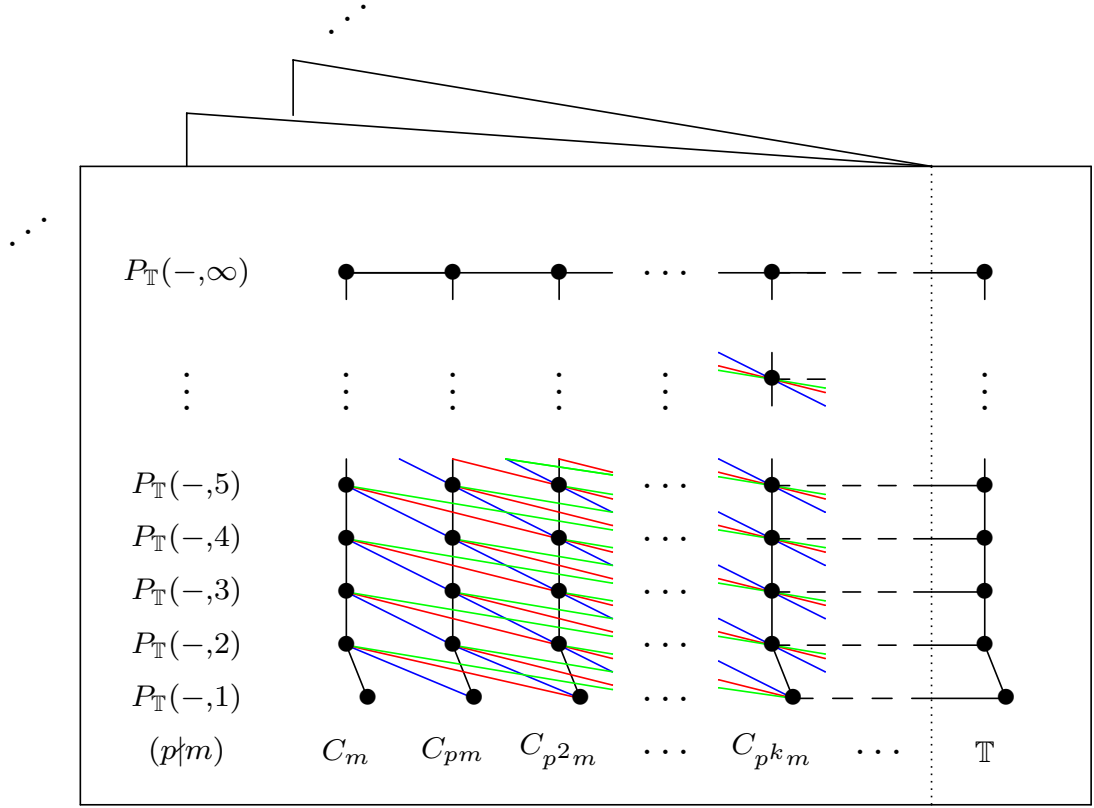

\[
\vcenter{\scalebox{1.4}{\xy
%
%
(-25,55)*{\iddots};
(5,75)*{\iddots};
%
{\ar@{-} (-20,0)*{};(-20,60)*{}};
{\ar@{-} (-20,60)*{};(60,60)*{}};
{\ar@{..} (60,60)*{};(60,0)*{}};
{\ar@{-} (60,0)*{};(-20,0)*{}};
{\ar@{-} (60,60)*{};(75,60)*{}};
{\ar@{-} (75,60)*{};(75,0)*{}};
{\ar@{-} (75,0)*{};(60,0)*{}};
{\ar@{-} (-10,65)*{};(60,60)*{}};
{\ar@{-} (-10,65)*{};(-10,60)*{}};
{\ar@{-} (0,70)*{};(60,60)*{}};
{\ar@{-} (0,70)*{};(0,64.8)*{}};
%
%
(5,5)*{\scriptstyle C_{m}};
{\ar@{-} (7,10)*{};(5,15)*{}};
{\ar@{-} (5,15)*{};(5,20)*{}};
{\ar@{-} (5,20)*{};(5,25)*{}};
{\ar@{-} (5,25)*{};(5,30)*{}};
{\ar@{-} (5,30)*{};(5,32.5)*{}};
{\ar@{-} (5,47.5)*{};(5,50)*{}};
%
%
{\ar@{-} (17,10)*{};(15,15)*{}};
{\ar@{-} (15,15)*{};(15,20)*{}};
{\ar@{-} (15,20)*{};(15,25)*{}};
{\ar@{-} (15,25)*{};(15,30)*{}};
{\ar@{-} (15,30)*{};(15,32.5)*{}};
{\ar@{-} (15,47.5)*{};(15,50)*{}};
%
%
{\ar@{-} (27,10)*{};(25,15)*{}};
{\ar@{-} (25,15)*{};(25,20)*{}};
{\ar@{-} (25,20)*{};(25,25)*{}};
{\ar@{-} (25,25)*{};(25,30)*{}};
{\ar@{-} (25,30)*{};(25,32.5)*{}};
{\ar@{-} (25,47.5)*{};(25,50)*{}};
%
%
{\ar@{-} (47,10)*{};(45,15)*{}};
{\ar@{-} (45,15)*{};(45,20)*{}};
{\ar@{-} (45,20)*{};(45,25)*{}};
{\ar@{-} (45,25)*{};(45,30)*{}};
{\ar@{-} (45,30)*{};(45,32.5)*{}};
{\ar@{-} (45,40)*{};(45,37.5)*{}};
{\ar@{-} (45,40)*{};(45,42.5)*{}};
{\ar@{-} (45,47.5)*{};(45,50)*{}};
%
%
{\ar@{-} (67,10)*{};(65,15)*{}};
{\ar@{-} (65,15)*{};(65,20)*{}};
{\ar@{-} (65,20)*{};(65,25)*{}};
{\ar@{-} (65,25)*{};(65,30)*{}};
{\ar@{-} (65,30)*{};(65,32.5)*{}};
{\ar@{-} (65,47.5)*{};(65,50)*{}};
%
%
{\ar@{-}@[blue] (17,10)*{};(5,15)*{}};
{\ar@{-}@[blue] (15,15)*{};(5,20)*{}};
{\ar@{-}@[blue] (15,20)*{};(5,25)*{}};
{\ar@{-}@[blue] (15,25)*{};(5,30)*{}};
{\ar@{-}@[blue] (15,30)*{};(10,32.5)*{}};
{\ar@{-} (15,50)*{};(5,50)*{}};
%
%
{\ar@{-}@[blue] (27,10)*{};(15,15)*{}};
{\ar@{-}@[blue] (25,15)*{};(15,20)*{}};
{\ar@{-}@[blue] (25,20)*{};(15,25)*{}};
{\ar@{-}@[blue] (25,25)*{};(15,30)*{}};
{\ar@{-}@[blue] (25,30)*{};(20,32.5)*{}};
{\ar@{-} (25,50)*{};(5,50)*{}};
%
%
{\ar@{-}@[blue] (30,12.5)*{};(25,15)*{}};
{\ar@{-}@[blue] (30,17.5)*{};(25,20)*{}};
{\ar@{-}@[blue] (30,22.5)*{};(25,25)*{}};
{\ar@{-}@[blue] (30,27.5)*{};(25,30)*{}};
{\ar@{-} (30,50)*{};(25,50)*{}};
%
%
{\ar@{-}@[blue] (47,10)*{};(40,13.5)*{}};
{\ar@{-}@[blue] (45,15)*{};(40,17.5)*{}};
{\ar@{-}@[blue] (45,20)*{};(40,22.5)*{}};
{\ar@{-}@[blue] (45,25)*{};(40,27.5)*{}};
{\ar@{-}@[blue] (45,30)*{};(40,32.5)*{}};
{\ar@{-}@[blue] (45,40)*{};(40,42.5)*{}};
{\ar@{-} (45,50)*{};(40,50)*{}};
%
%
{\ar@{-}@[blue] (50,12.5)*{};(45,15)*{}};
{\ar@{-}@[blue] (50,17.5)*{};(45,20)*{}};
{\ar@{-}@[blue] (50,22.5)*{};(45,25)*{}};
{\ar@{-}@[blue] (50,27.5)*{};(45,30)*{}};
{\ar@{-}@[blue] (50,37.5)*{};(45,40)*{}};
{\ar@{--}@[purple] (50,40)*{};(45,40)*{}};
{\ar@{-} (50,50)*{};(45,50)*{}};
%
%
{\ar@{-}@[red] (27,10)*{};(5,15)*{}};
{\ar@{-}@[red] (25,15)*{};(5,20)*{}};
{\ar@{-}@[red] (25,20)*{};(5,25)*{}};
{\ar@{-}@[red] (25,25)*{};(5,30)*{}};
{\ar@{-}@[red] (25,30)*{};(15,32.5)*{}};
%
%
{\ar@{-}@[red] (30,11.25)*{};(15,15)*{}};
{\ar@{-}@[red] (30,16.25)*{};(15,20)*{}};
{\ar@{-}@[red] (30,21.25)*{};(15,25)*{}};
{\ar@{-}@[red] (30,26.25)*{};(15,30)*{}};
{\ar@{-}@[red] (30,31.25)*{};(25,32.5)*{}};
%
%
{\ar@{-}@[red] (30,13.75)*{};(25,15)*{}};
{\ar@{-}@[red] (30,18.75)*{};(25,20)*{}};
{\ar@{-}@[red] (30,23.75)*{};(25,25)*{}};
{\ar@{-}@[red] (30,28.75)*{};(25,30)*{}};
%
%
{\ar@{-}@[red] (47,10)*{};(40,11.75)*{}};
{\ar@{-}@[red](45,15)*{};(40,16.25)*{}};
{\ar@{-}@[red] (45,20)*{};(40,21.25)*{}};
{\ar@{-}@[red] (45,25)*{};(40,26.25)*{}};
{\ar@{-}@[red] (45,30)*{};(40,31.25)*{}};
{\ar@{-}@[red] (45,40)*{};(40,41.25)*{}};
%
%
{\ar@{-}@[red] (50,13.75)*{};(45,15)*{}};
{\ar@{-}@[red] (50,18.75)*{};(45,20)*{}};
{\ar@{-}@[red] (50,23.75)*{};(45,25)*{}};
{\ar@{-}@[red] (50,28.75)*{};(45,30)*{}};
{\ar@{-}@[red] (50,38.75)*{};(45,40)*{}};
%
%
{\ar@{-}@[green] (30,10.83)*{};(5,15)*{}};
{\ar@{-}@[green] (30,15.83)*{};(5,20)*{}};
{\ar@{-}@[green] (30,20.83)*{};(5,25)*{}};
{\ar@{-}@[green] (30,25.83)*{};(5,30)*{}};
{\ar@{-}@[green] (30,31)*{};(20,32.5)*{}};
%
%
{\ar@{-}@[green] (30,12.5)*{};(15,15)*{}};
{\ar@{-}@[green] (30,17.5)*{};(15,20)*{}};
{\ar@{-}@[green] (30,22.5)*{};(15,25)*{}};
{\ar@{-}@[green] (30,27.5)*{};(15,30)*{}};
{\ar@{-}@[green] (30,31)*{};(20,32.5)*{}};
%
%
{\ar@{-}@[green] (30,14.17)*{};(25,15)*{}};
{\ar@{-}@[green] (30,19.17)*{};(25,20)*{}};
{\ar@{-}@[green] (30,24.17)*{};(25,25)*{}};
{\ar@{-}@[green] (30,29.17)*{};(25,30)*{}};
%
%
{\ar@{-}@[green] (47,10)*{};(40,11)*{}};
{\ar@{-}@[green] (45,15)*{};(40,15.83)*{}};
{\ar@{-}@[green] (45,20)*{};(40,20.83)*{}};
{\ar@{-}@[green] (45,25)*{};(40,25.83)*{}};
{\ar@{-}@[green] (45,30)*{};(40,30.83)*{}};
{\ar@{-}@[green] (45,40)*{};(40,40.83)*{}};
%
%
{\ar@{-}@[green] (50,14.17)*{};(45,15)*{}};
{\ar@{-}@[green] (50,19.17)*{};(45,20)*{}};
{\ar@{-}@[green] (50,24.17)*{};(45,25)*{}};
{\ar@{-}@[green] (50,29.17)*{};(45,30)*{}};
{\ar@{-}@[green] (50,39.17)*{};(45,40)*{}};
%
{\ar@{-}@[gray] (67,10)*{};(60,10)*{}};
{\ar@{-}@[gray] (65,15)*{};(60,15)*{}};
{\ar@{-}@[gray] (65,20)*{};(60,20)*{}};
{\ar@{-}@[gray] (65,25)*{};(60,25)*{}};
{\ar@{-}@[gray] (65,30)*{};(60,30)*{}};
{\ar@{-}@[gray] (65,50)*{};(60,50)*{}};
{\ar@{--}@[gray] (60,10)*{};(47,10)*{}};
{\ar@{--}@[gray] (60,15)*{};(45,15)*{}};
{\ar@{--}@[gray] (60,20)*{};(45,20)*{}};
{\ar@{--}@[gray] (60,25)*{};(45,25)*{}};
{\ar@{--}@[gray] (60,30)*{};(45,30)*{}};
{\ar@{--}@[gray] (60,50)*{};(45,50)*{}};
%
%
(7,10)*{\bullet};
(5,15)*{\bullet};
(5,20)*{\bullet};
(5,25)*{\bullet};
(5,30)*{\bullet};
(5,40)*{\vdots};
(5,50)*{\bullet};
%
%
(15,5)*{\scriptstyle C_{pm}};
(17,10)*{\bullet};
(15,15)*{\bullet};
(15,20)*{\bullet};
(15,25)*{\bullet};
(15,30)*{\bullet};
(15,40)*{\vdots};
(15,50)*{\bullet};
%
%
(25,5)*{\scriptstyle C_{p^2m}};
(27,10)*{\bullet};
(25,15)*{\bullet};
(25,20)*{\bullet};
(25,25)*{\bullet};
(25,30)*{\bullet};
(25,40)*{\vdots};
(25,50)*{\bullet};
%
%
(35,10)*{\hdots};
(35,15)*{\hdots};
(35,20)*{\hdots};
(35,25)*{\hdots};
(35,30)*{\hdots};
(35,40)*{\vdots};
(35,50)*{\hdots};
%
%
(35,5)*{\hdots};
%
%
(45,5)*{\scriptstyle C_{p^km}};
(47,10)*{\bullet};
(45,15)*{\bullet};
(45,20)*{\bullet};
(45,25)*{\bullet};
(45,30)*{\bullet};
(45,40)*{\bullet};
(45,50)*{\bullet};
%
%
(55,5)*{\hdots};
%
%
(65,5)*{\scriptstyle \T};
(67,10)*{\bullet};
(65,15)*{\bullet};
(65,20)*{\bullet};
(65,25)*{\bullet};
(65,30)*{\bullet};
(65,40)*{\vdots};
(65,50)*{\bullet};
%
%
(-10,10)*{\scriptstyle P_{\T}(-,1)};
(-10,15)*{\scriptstyle P_{\T}(-,2)};
(-10,20)*{\scriptstyle P_{\T}(-,3)};
(-10,25)*{\scriptstyle P_{\T}(-,4)};
(-10,30)*{\scriptstyle P_{\T}(-,5)};
(-10,40)*{\vdots};
(-10,50)*{\scriptstyle P_{\T}(-,\infty)};
%
(-10,5)*{\scriptstyle (p\nmid m)};
\endxy}}
\]
\caption{$\Spc(\SH_{\T,(p)}^c)$: This figure depicts the poset structure of $\Spc(\SH_{\T,(p)}^c)$, inclusions going from left to right and top to bottom. Each natural number $m$ that is not divisible by $p$ gives rise to a page as above, the different pages overlapping in the primes $P_{\T}(\T,n)$ of the full group $\T$. The inclusions $P_\T(C_{p^km},n+1)\subseteq P_\T(C_{p^{k+l}m},n)$ with $l>3$ are omitted in the picture, while the dashed lines indicate the inclusions $P_\T(C_{p^km},n)\subseteq P_{\T}(\T,n)$ for all $k\in \N$.}
\label{fig:spccircle}
\end{figure}

\begin{remark}\label{rem:o2} One might wonder whether upper semi-continuity is always implied by admissibility, but this is not the case in general. We consider the orthogonal group $O(2)$ for any prime~$p$ and the function 
\[ 
f\colon\Sub(O(2))/O(2) \to \NN 
\]
defined by
\[ 
f(C_n)=f(\T)=f(D_n)=1 \text{ and } f(O(2))=0. 
\]
Then $f$ is admissible, since no subgroup $K$ of $O(2)$ sits cotorally inside $O(2)$ and therefore $P_{O(2)}(K,1)\not\subseteq P_{O(2)}(O(2),1)$ by Part $ii)$ of \Cref{thm:inclusions}. But $f$ is not upper semi-continuous as the preimage of $[0,1)$ is not open.
\end{remark}

\section{Away from $|G|$} \label{sec:order}
Finally we explain that for a fixed compact Lie group $G$ the results in this paper determine the Balmer spectrum of $\Spc(\SH_{G,(p)}^c)$ for all but finitely many $p$. 

It is a result of tom Dieck \cite[Thm. 1]{tD77} that the set of orders of finite Weyl groups 
\[ 
\{|W_GH| \mid  H\subseteq G \text{ with $W_GH$ finite} \} 
\]
is bounded. We denote the least common multiple of this set by $|G|$ and call it the \emph{order of $G$}. Note that for finite $G$ this coincides with the usual definition of $|G|$. Then we have:
\begin{theorem} \label{thm:order} Let $G$ be a compact Lie group and $p$ be a prime not dividing the order $|G|$. Then a function $f\colon\Sub(G)/G\to \NN$ is admissible if and only if for all normal inclusions $K\subseteq H$ with $H/K$ an abelian $p$-toral group the inequality
\[ f(H)+rk_p(\pi_0(H/K)) \geq f(K) \]
holds.
\end{theorem}
In other words, if $p$ does not divide $|G|$, there is an inclusion
\[ 
P_G(K,n+i)\subseteq P_G(H,n) 
\]
if and only if $K$ is $G$-conjugate to a normal subgroup $K^g$ of $H$ with $H/K^g$ an abelian $p$-toral group such that $i\geq rk_p(\pi_0(H/K^g))$. In particular, this determines the Balmer spectrum of finite $G$-spectra away from the order $|G|$. Note that for finite groups $G$ the statement of \Cref{thm:order} follows directly from \Cref{thm:inclusions}, since for $p$ not dividing $|G|$ there exist no proper $p$-subcotoral inclusions $K\subseteq H$. For general compact Lie groups we still need to take care of the maximal torus and its $p$-toral subgroups and hence rely on understanding the inclusions for abelian groups as worked out in the previous section.

\begin{proof}[Proof of \Cref{thm:order}] We first show that when $p$ does not divide the order $|G|$ and $K\subseteq H$ is a $p$-subcotoral inclusion of closed subgroups of $G$, then $K$ is normal in $H$ and the quotient $H/K$ is an abelian $p$-toral group. To see this we first note that when $K\subseteq H$ is $p$-subcotoral, then $K$ and $H$ in particular have the same $p$-perfection inside $G$, i.e., $\omega_G(\mathcal{O}(K))= \omega_G(\mathcal{O}(H))$. We recall from \cite[Thm. 4]{tD75} that the set of conjugacy classes of subgroups of $G$ sharing this $p$-perfection with $H$ and $K$ has a unique maximal element $M$, which is characterized by the property that its Weyl group is finite and of order coprime to $p$. Without loss of generality, we can choose an $M$ which contains $H$. Since $\mathcal{O}(M)$ has the same $p$-perfection as $M$ and is contained inside $M$ with finite $p$-power index, it follows that $\mathcal{O}(M)$ must have finite Weyl group (since otherwise $\omega_G(\mathcal{O}(M))$ could not be contained in a subgroup conjugate to $M$). Now, by assumption $|W_G\mathcal{O}(M)|$ is coprime to $p$, and thus $M$ must equal $\mathcal{O}(M)$. But this implies that $H$ is contained in $\mathcal{O}(M)= \omega_G(\mathcal{O}(K))$. Thus, there is a chain of inclusions
\[ 
\mathcal{O}(K)\subseteq K \subseteq H \subseteq \omega_G(\mathcal{O}(K)). 
\]
Since $\mathcal{O}(K)\subseteq \omega_G(\mathcal{O}(K))$ is a cotoral inclusion, we find that $K$ must be normal in $H$ with abelian quotient. Since $K$ is further $p$-subcotoral in $H$, the quotient must be a $p$-toral compact Lie group as desired.

Now we let $K$ and $H$ be arbitrary closed subgroups of $G$ and show the statement of the theorem, i.e., that there is an inclusion $P_G(K,n+i)\subseteq P_G(H,n)$ if and only if $K$ is $G$-conjugate to a normal subgroup $K^g$ of $H$ with $H/K^g$ an abelian $p$-toral group such that $i\geq rk_p(\pi_0(H/K^g))$. By \Cref{prop:changeofgroups} we know that there is an inclusion $P_G(K,n+i)\subseteq P_G(H,n)$ if and only if there is an inclusion $P_H(K^g,n+i)\subseteq P_H(H,n)$ for some conjugate $K^g$ contained in $H$. \Cref{thm:inclusions} shows that the latter inclusion can only happen when $K^g\subseteq H$ is $p$-subcotoral, which as we just saw implies that $K^g\subseteq H$ is normal with abelian $p$-toral quotient. Applying \Cref{prop:changeofgroups} once more gives that in this case there is an inclusion $P_H(K^g,n+i)\subseteq P_H(H,n)$ if and only if there is an inclusion $P_{H/K^g}(K^g/K^g,n+i)\subseteq P_{H/K^g}(H/K^g,n)$. Since the group $H/K^g$ is abelian $p$-toral, \Cref{thm:abelian} implies that the latter inclusion happens if and only if $i\geq rk_p(\pi_0(H/K^g))$, which finishes the proof.
\end{proof}

\appendix
\section{}
In this appendix we give the proof of \Cref{lem:induced}, the statement of which we recall for convenience:

\begin{lemma} Let $H$ be a closed subgroup of a compact Lie group $G$, and let $X$ be a finite $H$-spectrum. Then the following hold:
\begin{enumerate}[i)]
	\item The type of $\Phi^H(\Ind_H^G(X))$ is equal to the type of $\Phi^H(X)$.
	\item If $K$ is another closed subgroup of $G$, then the type of $\Phi^K(\Ind_H^G(X))$ is at least the minimum of the types of $\Phi^{K'}(X)$ for all subgroups $K'$ of $H$ that are conjugate to $K$ in $G$.
	\item If $K$ is not subconjugate to $H$, then the geometric fixed points $\Phi^K(\Ind_H^G(X))$ are trivial.
\end{enumerate}
\end{lemma}
\begin{remark} \label{rem:subconjugacy} 
Before we prove the lemma, we note that the set of $H$-conjugacy classes of such subgroups $K'$ is always finite, by the following argument: For an element $g\in G$, the conjugate $K^g=g^{-1}Kg$ is a subgroup of $H$ if and only if $gH$ is a $K$-fixed point in $G/H$. Moreover, if two elements $g$ and $\widetilde{g}$ represent the same coset in this $K$-fixed space, then the conjugates $K^g$ and $K^{\widetilde{g}}$ are $H$-conjugate. Hence one obtains a well-defined surjective map
\begin{align*} (G/H)^K &\to \{ \text{$H$-conjugacy classes of subgroups $K'\subseteq H$ conjugate to $K$}\}\\
gH & \mapsto [K^g].
\end{align*}
This map is not injective, but it becomes so after quotienting out by the residual action of the normalizer $N_G(K)$ on $(G/H)^K$. Indeed, if $n\in N_G(K)$, then $K^{ng}=(K^n)^g=K^g$, hence the map factors through the quotient, and if $K^g=(K^{\widetilde{g}})^h$ for some $h\in H$, then $\widetilde{g}hg^{-1}\in N_G(K)$ and $(\widetilde{g}hg^{-1})gH=\widetilde{g}H$. Now, the set $(G/H)^K/N_G(K)$ is finite as a consequence of the Montgomery--Zippin theorem, see \cite[Cor. II.5.7]{Bre72}. We note also that this implies that the $N_G(K)$-orbits of $(G/H)^K$ are open, and hence $(G/H)^K$ decomposes as the topological disjoint union over these.
\end{remark}
\begin{proof}[Proof of \Cref{lem:induced}] Let $K$ be a closed subgroup of $G$, later in the proof we check what happens in the special case where $K=H$. We start by showing the claims for suspension spectra. Let $A$ be a (not neccesarily finite) based $H$-CW complex and consider the based $G$-space $G_+\wedge_H A$, i.e., the space level induction of $A$ to a $G$-space. We want to understand the $K$-fixed points $(G_+\wedge_H A)^K$. Let $g_1,\hdots,g_s\in G$ be a set of representatives of the $N_G(K)$-orbits of $(G/H)^K$, as in the above remark. Note that if $n\in N_G(K)$ and $a\in A^{K^{g_i}}$, then the class $[(ng_i,a)]\in G_+\wedge_H A$ is $K$-fixed, since 
\[ [(kng_i,a)]=[(ng_i(g_i^{-1}n^{-1}kng_i),a)]=[(ng_i,(g_i^{-1}n^{-1}kng_i)a)]=[(ng_i,a)]. \]
The last step uses that $g_i^{-1}n^{-1}kng_i\in K^{g_i}$, because $n$ normalizes $K$. Moreover, if $\widetilde{n}\in N_H(K^{g_i})$, then $g_i\widetilde{n}g_i^{-1}\in N_G(K)$ and
\[ 
[(ng_i,\widetilde{n}a)]=[(ng_i\widetilde{n},a)]=[(n(g_i\widetilde{n}g_i^{-1})g_i,a)]. 
\]
Thus we obtain a continuous map
\begin{align*} \bigvee_{g_1,\hdots,g_s} (N_G(K)\cdot g_i)_+\wedge_{N_H(K^{g_i})} A^{K^{g_i}} & \to  (G_+\wedge_H A)^K\\
[(ng_i,a)] & \mapsto [(ng_i,a)]
\end{align*}
with $N_H(K^{g_i})$ acting on $N_G(K)\cdot g_i$ by right multiplication. We leave it to the reader to check that this map is bijective. It is then not hard to show that it is a homeomorphism, using that $(G/H)^{K}$ decomposes as a topological disjoint union of its $N_G(K)$-orbits as explained in \Cref{rem:subconjugacy} (see also the proof of \cite[Prop. B.17]{Sch18}). Thus, in summary, for based $H$-spaces $A$ there is a natural homeomorphism
\[ 
(G_+\wedge_H A)^K\cong \bigvee_{g_1,\hdots,g_s} (N_G(K)\cdot g_i)_+\wedge_{N_H(K^{g_i})} A^{K^{g_i}}, 
\]
where the $g_i$ are elements of $G$ such that $K^{g_1},\hdots,K^{g_s}$ are representatives of $H$-conjugacy classes of subgroups of $H$ that are conjugate to $K$ in $G$. Hence, using that geometric fixed points and induction commute with suspension spectra (\Cref{sec:geometricfixed} and \Cref{sec:restriction}), we obtain an equivalence
\[ 
\Phi^K(\Ind_H^G(\Sigma^{\infty} A))\simeq \bigvee_{g_1,\hdots,g_s} \Sigma^{\infty} ((N_G(K)\cdot g_i)_+\wedge_{N_H(K^{g_i})} A^{K^{g_i}}). 
\]
Now we assume that (the $p$-localization of) $\Sigma^{\infty}A$ is a finite $p$-local $H$-spectrum. Then, in order to prove Part $ii)$ we need to show that the type of $(N_G(K)\cdot g_i)_+\wedge_{N_H(K^{g_i})} A^{K^{g_i}}$ is at least the type of $A^{K^{g_i}}$. This follows from the fact that the $N_H(K^{g_i})$-action on $N_G(K)$ is free, hence $N_G(K)$ is $N_H(K^{g_i})$-equivariantly built out of cells of the form $N_H(K^{g_i})_+\wedge D^n$. Therefore, $\Sigma^{\infty}(N_G(K)\cdot g_i)_+\wedge_{N_H(K^{g_i})} A^{K^{g_i}}$ lies in the thick subcategory generated by $\Sigma^{\infty} (N_H(K^{g_i})_+\wedge_{N_H(K^{g_i})} A^{K^{g_i}})\simeq \Sigma^{\infty} A^{K^{g_i}}$, hence its type is at least as large. Furthermore, if no conjugate of $K$ is a subgroup of $H$, then $(G_+\wedge_H A)^K$ consists of a point, hence $\Phi^K(\Ind_H^G(\Sigma^{\infty} A))$ is trivial, proving Part $iii)$ for suspension spectra.

\begin{remark} \label{rem:finequal} If $G$ is finite, then $N_G(K)\cdot g_i$ is discrete,  so picking a set of representatives of $N_H(K^{g_i})$-orbits gives rise to a homeomorphism
\[ (N_G(K)\cdot g_i)_+\wedge_{N_H(K^{g_i})} A^{K^{g_i}}\cong (N_G(K)\cdot g_i)/N_H(K^{g_i})_+\wedge A^{K^{g_i}}, \]
and hence the type of $(N_G(K)\cdot g_i)_+\wedge_{N_H(K^{g_i})} A^{K^{g_i}}$ is in fact equal to that of $A^{K^{g_i}}$. 
Such an argument does not work for general compact Lie groups, because the projection $N_G(K)\cdot g_i\to (N_G(K)\cdot g_i)/N_H(K^{g_i})$ need not split. See \Cref{rem:counter} for a counterexample.
\end{remark}
When $K$ is equal to $H$, the formula above consists of only one summand corresponding to $H$ itself, since $H$ is not conjugate to a proper subgroup of itself (see the proof of \Cref{lem:primerelations}). As $N_H(H)=H$, this summand is given by $\Sigma^{\infty} N_G(H)_+\wedge_H A^H$. The $H$-action on $A^H$ is trivial, hence $N_G(H)_+\wedge_H A^H$ is homeomorphic to $(N_G(H)/H)_+\wedge A^H$, whose type equals that of $A^H$ since it contains $A^H$ as a retract (and on the other hand lies in the thick subcategory generated by $A^H$). Therefore, the type of $\Phi^H(\Ind_H^G(\Sigma^{\infty}A))$ equals the type of $\Phi^H(\Sigma^{\infty}A)$, as desired. This finishes the proof of the lemma for finite $G$-spectra of the form $\Sigma^{\infty}A$.

Now, by \Cref{lem:finiteg}, a general finite $H$-spectrum $X$ is of the form $S^{-V}\wedge \Sigma^{\infty} A$ for $A$ a homotopy retract of the $p$-localization of a finite $H$-CW complex and an $H$-representation $V$. By a standard fact in representation theory of compact Lie groups (see for example \cite[Thm. III.4.5]{BtD85}) there exists a $G$-representation $W$ such that $V$ sits inside the restriction $\res^G_H W$ of $W$ to an $H$-representation. Then
\[ 
X\simeq S^{-V}\wedge \Sigma^{\infty} A\simeq (\res_H^G S^{-W})\wedge \Sigma^{\infty} (S^{(\res_H^GW)-V}\wedge A). 
\]
Let $A'=S^{(\res_H^GW)-V}\wedge A$, which is again a homotopy retract of the $p$-localization of a finite $H$-CW complex and hence $\Sigma^{\infty}A'$ is a finite $p$-local $H$-spectrum. Then $X\simeq \res_H^G S^{-W}\wedge \Sigma^{\infty} A'$ and by Frobenius reciprocity (see \Cref{sec:restriction}) there is an equivalence
\[ 
\Ind_H^G(X)\simeq \Ind_H^G((\res_H^G S^{-W})\wedge \Sigma^{\infty} A')\simeq S^{-W}\wedge \Ind_H^G(\Sigma^{\infty} A'). 
\]
Note that smashing with a representation sphere has no effect on the type of the geometric fixed points for any subgroup. Hence, for a closed subgroup $K$ of $G$ we obtain
\begin{align*} \type{K}(\Ind_H^G(X)))  &=\type{K}(\Ind_H^G(\Sigma^{\infty} A')) \\
&\geq \min(\type{K'}(\Sigma^{\infty}A')\ |\ K'\subseteq H, K'\sim K)  \\
& =\min(\type{K'}(X)\ |\ K'\subseteq H, K'\sim K), \end{align*}
as desired, with equality in the case $K=H$. Finally, if $K$ is not conjugate to a subgroup of $H$, then
\[ \Phi^K(\Ind_H^G(X))\simeq \Phi^K(S^{-W})\wedge \Phi^K(\Ind_H^G(\Sigma^{\infty} A'))\simeq *, \]
since $\Phi^K(\Ind_H^G(\Sigma^{\infty} A'))\simeq *$.
\end{proof}

\bibliographystyle{amsalpha}
\bibliography{literature}

\end{document}